\documentclass[addpoints,11pt,letter]{article}

\usepackage{amsfonts}
\usepackage{amsmath}
\usepackage{amssymb}
\usepackage{amsthm}
\usepackage{authblk}
\usepackage{chngpage}
\usepackage{color}
\usepackage{commath}
\usepackage{comment}
\usepackage[shortlabels]{enumitem}
\usepackage{esint}
\usepackage{fge}
\usepackage{graphics}
\usepackage{graphicx}
\usepackage[hidelinks]{hyperref}
\usepackage{latexsym}
\usepackage{mathrsfs}
\usepackage{mathtools}
\usepackage{multicol}
\usepackage{multirow}
\usepackage{nccmath}
\usepackage{pgfplots}
\usepackage{textcomp}
\usepackage[sc]{titlesec}
\usepackage{tikz}
\usepackage[utf8x]{inputenc}

\usepackage{lmodern}
\usepackage[T1]{fontenc}

\DeclareMathOperator{\glt}{GLT}
\DeclareMathOperator{\glteq}{\sim_{\glt}}

\DeclareMathOperator{\SL}{SL}

\DeclareSymbolFont{eulargesymbols}{U}{zeuex}{m}{n}
\DeclareMathSymbol{\intop}{\mathop}{eulargesymbols}{"52}
\DeclareMathSymbol{\ointop}{\mathop}{eulargesymbols}{"49}

\newcommand{\al}{\alpha}
\newcommand{\bC}{\mathbb C}

\newcommand{\be}{\beta}

\newcommand{\bN}{\mathbb N}

\newcommand{\bR}{\mathbb R}

\newcommand{\bT}{\mathbb T}
\newcommand{\bZ}{\mathbb Z}

\newcommand{\cI}{\mathcal I}

\newcommand{\de}{\delta}
\newcommand{\De}{\Delta}

\newcommand{\e}{\mathrm{e}}

\newcommand{\eps}{\varepsilon}
\newcommand{\ga}{\gamma}
\newcommand{\Ga}{\Gamma}
\newcommand{\ka}{\kappa}
\newcommand{\la}{\lambda}

\newcommand{\nf}{\infty}

\newcommand{\ph}{\varphi}
\newcommand{\Ph}{\Phi}
\newcommand{\si}{\sigma}

\newcommand{\x}{\raisebox{0.2mm}{\mbox{\tiny $\times$}\hspace{-0.3mm}}}
\newcommand{\tht}{\theta}
\newcommand{\Tht}{\Theta}

\renewcommand{\ge}{\geqslant}
\renewcommand{\d}{\dif}

\renewcommand{\i}{\mathrm{i}}

\renewcommand{\le}{\leqslant}

\def\Yint#1{\mathchoice
{\YYint\displaystyle\textstyle{#1}}
{\YYint\textstyle\scriptstyle{#1}}
{\YYint\scriptstyle\scriptscriptstyle{#1}}
{\YYint\scriptscriptstyle\scriptscriptstyle{#1}}\!\int}
\def\YYint#1#2#3{{\setbox0=\hbox{$#1{#2#3}{\int}$}
\vcenter{\hbox{$#2#3$}}\kern-.52\wd0}}
\def\dint{\,\Yint\nmid}

\usetikzlibrary{arrows,shapes,positioning}
\usetikzlibrary{decorations.markings}
\tikzstyle arrowstyle=[scale=1]
\tikzstyle directed=[postaction={decorate,decoration={markings,
    mark=at position .65 with {\arrow[arrowstyle]{stealth}}}}]
\tikzstyle reverse directed=[postaction={decorate,decoration={markings,
    mark=at position .65 with {\arrowreversed[arrowstyle]{stealth};}}}]

\numberwithin{equation}{section}

\newtheorem{lemma}{Lemma}[section]
\newtheorem{theorem}[lemma]{Theorem}

\newtheorem{definition}{Definition}[section]

\title{\textsc{Eigenvalue superposition expansion for Toeplitz matrix-sequences, generated by linear combinations of matrix-order dependent symbols, and applications to fast eigenvalue computations}}

\author[1]{M. Bogoya\thanks{johanmanuel.bogoya@uninsubria.it}}
\author[1]{S. Serra--Capizzano\thanks{s.serracapizzano@uninsubria.it}}

\affil[1]{\footnotesize University of Insubria, Como, Italy.\\
Dipartimento di Scienza e Alta Tecnologia.}

\date{\today}

\begin{document}

\maketitle

\begin{abstract}
The eigenvalues of Toeplitz matrices $T_{n}(f)$ with a real-valued symbol $f$, satisfying some conditions and tracing out a simple loop over the interval $[-\pi,\pi]$, are known to admit an asymptotic expansion with the form
\[
\la_{j}(T_{n}(f))=f(d_{j,n})+c_{1}(d_{j,n})h+c_{2}(d_{j,n})h^{2}+O(h^{3}),
\]
where $h=\frac{1}{n+1}$, $d_{j,n}=\pi j h$, and $c_k$ are some bounded coefficients depending only on $f$. The numerical results presented in the literature suggests that the effective conditions for the expansion to hold are weaker and reduce to an even character of $f$, to a fixed smoothness, and to its monotonicity over $[0,\pi]$. \\
In this note we investigate the superposition caused over this expansion, when considering a linear combination of symbols that is
\[
\la_{j}\big(T_{n}(f_0)+\be_{n}^{(1)} T_{n}(f_{1}) + \be_{n}^{(2)} T_{n}(f_{2}) +\cdots\big),
\]
where $ \be_{n}^{(t)}=o\big(\be_{n}^{(s)}\big)$ if $t>s$ and the symbols $f_{j}$ are either simple loop or satisfy the weaker conditions mentioned before.
We prove that the asymptotic expansion holds also in this setting under mild assumptions and we show numerically that there is much more to investigate, opening the door to linear in time algorithms for the computation of eigenvalues of large matrices of this type.

The problem is of concrete interest in particular in the case where the coefficients of the linear combination are functions of $h$, considering spectral features of matrices stemming from the numerical approximation of standard differential operators and distributed order fractional differential equations, via local methods such as Finite Differences, Finite Elements, Isogeometric Analysis etc.
\medskip\\
\noindent\textbf{MSC Classes}: Primary 15B05, 65F15, 47B35. Secondary 15A18, 47A38.
\end{abstract}

\section{Introduction and Preliminaries}

The current section is divided into two parts, both of introductory type.

In the first (Subsection \ref{ssec:prelim}), we introduce notations, definitions, and preliminary results concerning Toeplitz structures, which are essential for the mathematical formulation of the problem and its technical solution.

In the second (Subsection \ref{ssec:intro}), we present a standard introduction to the considered problem and of the previous literature. Beside the technical results, which are mathematically non trivial, the new findings have application to the design of fast eigensolvers (of matrix-less type; see \cite{EkGa19,EkGa18} and references therein) for the computation in linear time of the eigenvalues of large matrices, stemming e.g. from the numerical approximation via local methods, like Finite Differences, Finite Elements, Isogeometric Analysis etc (see \cite{Ci02,CoHu09,St89} and references there reported), of coercive differential equations like diffusion-advection, or from distributed order fractional differential equations again approximated by using local methods (see \cite{Ab19,BoGr21,MaSe21} and references therein).
In Subsection \ref{ssec:intro} we also present a brief account on the theory of Generalized Locally Toeplitz ($\glt$) matrix-sequences \cite{BaGa20b,BaGa20a,GaSe17,GaSe18}, which is of support in interpreting our findings from a different perspective.

The rest of the paper is organized as follows. In Section \ref{sec:Main} we present the problem and the results and in Section \ref{sec:Main-Proof} the related proofs are given. Section \ref{sec:Num} is devoted to numerical experiments, showing the potential of matrix-less eigensolvers. Finally,
in Section \ref{sec:Final}, conclusions, open problems, and future promising developments are illustrated, including the challenging connections with the notion of $\glt$ momentary symbols (see \cite{BoEk21-0,BoEk21a}).

\subsection{Preliminaries}\label{ssec:prelim}
We consider a Toeplitz matrix
\begin{equation*}
\left[a_{i-j}\right]_{i,j=1}^{n}=\begin{bmatrix}
a_0 & a_{-1} & a_{-2} & \cdots & \cdots & a_{-(n-1)} \\
a_{1} & \ddots & \ddots & \ddots & & \vdots\\
a_{2} & \ddots & \ddots & \ddots & \ddots & \vdots\\
\vdots & \ddots & \ddots & \ddots & \ddots & a_{-2}\\
\vdots & & \ddots & \ddots & \ddots & a_{-1}\\
a_{n-1} & \cdots & \cdots & a_{2} & a_{1} & a_0
\end{bmatrix},
\end{equation*}
which is characterized from the fact to show constant entries along each diagonal. When taking a function $f\colon[-\pi,\pi]\to\bC$ belonging to $L^{1}([-\pi,\pi])$, the $n$th Toeplitz matrix generated by $f$ is formally expressed as
\begin{equation*}
T_{n}(f)=\big[\hat f_{i-j}\big]_{i,j=1}^{n},
\end{equation*}
where the quantities $\hat f_k$ are the Fourier coefficients of $f$,
\begin{equation*}
\hat f_{k}=\frac1{2\pi}\int_{-\pi}^{\pi}f(\tht)\,\e^{-\i k\tht}\d \tht,\qquad k\in\bZ.
\end{equation*}
In order to fix the terminology, we refer to $\{T_{n}(f)\}_{n}$ as the Toeplitz sequence generated by $f$, which in turn is called the generating function or the symbol of $\{T_{n}(f)\}_{n}$: we remind that the notion of symbol is a different notion recalled in Definition \ref{def-distribution} and the related theorem for Toeplitz matrix-sequences is reported in Theorem \ref{toeplitz distribution}. If the generating function $f$ is real-valued, then all the matrices $T_{n}(f)$ are Hermitian and their spectral properties are known in detail, from the localization of the eigenvalues to the asymptotic spectral distribution in the Weyl sense; see Definition \ref{def-distribution},
Theorem \ref{toeplitz distribution}, \cite{BoSi99,GaSe17} and the references therein.

In the current work we focus on the case where $f$ is real-valued and shows an infinite cosine expansion, that is, a function of the form
\[
f(\tht)=\hat f_0+2\sum_{k=1}^m\hat f_k\cos(k\tht),\qquad \hat f_0,\hat f_{1},\ldots,\hat f_m\in\bR,\qquad m\in\bN\cup \{\nf\},
\]
so that $f(2\pi-s)=f(s)$. We say that a cosine expansion function $f$ is monotone if it is either increasing or decreasing over the interval $[0,\pi]$.
The $n$th Toeplitz matrix generated by $f$ is the real symmetric matrix given by
\[ T_{n}(f) = \left[\begin{array}{cc|ccccccc|cc}
\hat f_0 & \multicolumn{1}{c}{\hat f_{1}} & \cdots & \hat f_m & & & & & \multicolumn{1}{c}{} & & \\
\hat f_{1} & \multicolumn{1}{c}{\ddots} & \ddots & & \ddots & & & & \multicolumn{1}{c}{} & & \\
\vdots & \multicolumn{1}{c}{\ddots} & \ddots & \ddots & & \ddots & & & \multicolumn{1}{c}{} & & \\
\hat f_m & \multicolumn{1}{c}{} & \ddots & \ddots & \ddots & & \ddots & & \multicolumn{1}{c}{} & & \\
& \multicolumn{1}{c}{\ddots} & & \ddots & \ddots & \ddots & & \ddots & \multicolumn{1}{c}{} & & \\
\cline{3-9}
& & {\hat f_m}^{\vphantom{\int}} & \cdots & \hat f_{1} & \hat f_0 & \hat f_{1} & \cdots & \hat f_m & & \\
\cline{3-9}
& \multicolumn{1}{c}{} & & \ddots & & \ddots & \ddots & \ddots & \multicolumn{1}{c}{} & \ddots & \\
& \multicolumn{1}{c}{} & & & \ddots & & \ddots & \ddots & \multicolumn{1}{c}{\ddots} & & \hat f_m\\
& \multicolumn{1}{c}{} & & & & \ddots & & \ddots & \multicolumn{1}{c}{\ddots} & \ddots & \vdots\\
& \multicolumn{1}{c}{} & & & & & \ddots & & \multicolumn{1}{c}{\ddots} & \ddots & \hat f_{1}\\
& \multicolumn{1}{c}{} & & & & & & \hat f_m & \multicolumn{1}{c}{\cdots} & \hat f_{1} & \hat f_0
\end{array}\right],\]
with $m\le n-1$, and we emphasize that our analysis will be not restricted to banded matrices, as done mainly in some of the previous works (see for instance \cite{EkFu18b,EkGa19,EkGa18}).
In \cite{BoBo15a,BoGr17,BoGr10} the subsequent results were proven: if $f$ is monotone in the sense specified before, smooth, and satisfies certain additional assumptions, which include the requirements that $f'(\tht)\ne0$ for $\tht\in(0,\pi)$ and $f''(\tht)\ne0$ for $\tht\in\{0,\pi\}$, then, for every integer $\al\ge0$, every $n$, and every $j=1,\ldots,n$, the representation below
\begin{equation}\label{hoapp}
\la_{j}(T_{n}(f))=f(\tht_{j,n})+\sum_{k=1}^{\al}c_k(\tht_{j,n})h^k+E_{j,n,\al},
\end{equation}
is true, where the related asymptotic expansion has the following features:
\begin{itemize}
\item the eigenvalues of $T_{n}(f)$ are arranged in nondecreasing or nonincreasing order, depending on whether $f$ is increasing or decreasing;
\item $\{c_k\}_{k=1,2,\ldots}$ is a sequence of functions from $[0,\pi]$ to $\bR$ which depends only on $f$;
\item $h=\frac{1}{n+1}$ and $\tht_{j,n}=\frac{j\pi}{n+1}=j\pi h$;
\item $E_{j,n,\al}=O(h^{\al+1})$ is the remainder (the error), satisfying the inequality $|E_{j,n,\al}|\le C_\al h^{\al+1}$ for some constant $C_\al$ depending only on $\al$ and $f$.
\end{itemize}

When approximating the operator $(-1)^q\frac{\d^{\,2q}}{\d x^{2q}}$, $q=0,1,2,\ldots$, on a given interval with proper boundary conditions, we end up with structures either as $T_{n}(f_q)$ with $f_q$ being a monotone, real-valued cosine polynomial of the form
\begin{equation}\label{f_q}
f_q(\tht)=(2-2\cos\tht)^q,\qquad q=0,1,2,\ldots,
\end{equation}
if one uses centered Finite Differences of precision order $2$ \cite{St89} or as $T_{n}(g_q)$ with $g_q$ being a real-valued cosine polynomial of the form
\begin{equation}\label{g_q}
g_q(\tht)=(2-2\cos\tht)^q p_q(\tht),\qquad q=0,1,2,\ldots,
\end{equation}
where $p_q$ is a strictly positive cosine polynomial, when using the Isogeometric Analysis with maximal regularity \cite{CoHu09,GaMa14}.
We recall that the considered Finite Differences are characterized by $O(h^{2})$ precision and minimal bandwidth, while in the case of (\ref{g_q}) the bandwidth is larger, but a much higher precision is described in \cite{CoHu09}.

Unfortunately, for these generating functions the requirement that $f''(0)\ne0$ is not satisfied if $q\ne2$ and in \cite{EkGa18} numerical evidences, that the higher order approximation~\eqref{hoapp} holds even in this ``degenerate case'', are reported and discussed. Actually, based on numerical experiments, it was conjectured that \eqref{hoapp} holds at least for all monotone cosine trigonometric polynomials $f$.

In \cite{BoBo15a}, the authors also briefly mentioned that the asymptotic expansion~\eqref{hoapp} can be used to compute an accurate approximation of $\la_{j}(T_{n}(f))$ for very large $n$, provided the values $\la_{j_{1}}(T_{n_{1}}(f))$, $\la_{j_{2}}(T_{n_{2}}(f))$, $\la_{j_{3}}(T_{n_{3}}(f))$ are available for moderately sized $n_{1},n_{2},n_{3}$ with $\tht_{j_{1},n_{1}}=\tht_{j_{2},n_{2}}=\tht_{j_{3},n_{3}}=\tht_{j,n}$. The idea has evolved in \cite{EkFu18b,EkGa19,EkGa18} and highly accurate matrix-less methods of optimal linear cost have been developed: a wide set of numerical experiments has been reported, accompanied by an appropriate error analysis.
It should be stressed that in essence the matrix-less algorithms are completely analogous to the extrapolation procedure, which is employed in the context of Romberg integration for obtaining high precision approximations of an integral from a few coarse trapezoidal approximations \cite[\S3.4]{StBu02}. In this regard, the asymptotic expansion~\eqref{hoapp} plays here the same role as the Euler--Maclaurin summation formula \cite[\S3.3]{StBu02}.

Here we are interested to extend the machinery both theoretically and computationally to the more involved case of linear combinations of matrix-order depending symbols (see (\ref{eigs of superposition}) in Subsection \ref{ssec:intro}), with the following two targets:
\begin{itemize}
\item proving formally the asymptotic expansions;
\item giving related matrix-less procedures.
\end{itemize}

\subsection{The Problem and the Literature}\label{ssec:intro}

The eigenvalues of Toeplitz matrices $T_{n}(f)$ with a real-valued symbol $f$, satisfying some conditions and tracing out a simple loop over the interval $[-\pi,\pi]$, are known to admit an asymptotic expansion with the form
\[
\la_{j}(T_{n}(f))=f(d_{j,n})+c_{1}(d_{j,n})h+c_{2}(d_{j,n})h^{2}+O(h^{3}),
\]
where $h=\frac{1}{n+1}$, $d_{j,n}=\pi j h$, and $c_k$ are some bounded coefficients depending only on $f$. The numerical results presented in the literature suggests that the effective conditions for the expansion to hold are weaker and reduce to an even character of $f$ and monotonicity over $[0,\pi]$. \\
As already mentioned in Subsection \ref{ssec:prelim}, we investigate the superposition caused over this expansion, when considering a linear combination of symbols, that is
\begin{equation}\label{eigs of superposition}
\la_{j}\big(T_{n}(f_0+\be_{n}^{(1)} f_{1} + \be_{n}^{(2)} f_{2} +\cdots)\big),
\end{equation}
where $ \be_{n}^{(t)}=o(\be_{n}^{(s)})$ if $t>s$ and the symbols $f_{j}$ are either simple loop or satisfy the weaker conditions mentioned before.
We prove that the asymptotic expansion holds also in this setting under mild assumptions and we show numerically that there is much more to investigate, opening the door to linear in time algorithms for the computation of eigenvalues of large matrices of this type.

The problem is of application interest in particular in the case where the coefficients of the linear combination are given functions of $h$, considering spectral features of matrices stemming from the numerical approximation of standard differential operators and distributed order fractional differential equations \cite{BoGr21,DoMa16,MaSe21}. In particular, for standard differential operators, our approach could be very promising for any local approximation technique of integro-differential operators, giving rise to $\glt$ matrix-sequences (see \cite{BaGa20b,BaGa20a,GaSe17,GaSe18} and references therein), with special attention to the case of Finite Elements \cite{Ci02} and Isogeometric analysis both with maximal regularity and intermediate regularity \cite{CoHu09}.

Furthermore, from a theoretical viewpoint, it is worth stressing that in the current work it is the first time that an eigenvalue expansion is theoretically obtained for a Toeplitz matrix-sequence with a symbol depending on $n$.

In the next steps we introduce the essential of the $\glt$ theory (see \cite{BaGa20b,BaGa20a,GaSe17,GaSe18} and references therein) and of the new concept of $\glt$ momentary symbols \cite{BoEk21a}, related to matrix structures as those appearing in (\ref{eigs of superposition}).

\subsubsection{$\glt$ theory and $\glt$ momentary symbols}

In this technical part we give the essentials of the $\glt$ theory. We start with the definition of spectral symbol and of symbol (in the singular value sense). Then we give the axioms that characterize the $\glt$ matrix-sequences and we spend few words on the new concept of $\glt$ momentary symbols.

\begin{definition}\label{def-distribution}
Let $f\colon D\to\bC$ be a measurable function defined on the Lebesgue measurable set $D$ of positive and finite measure. Assume that $\{A_{n}\}_{n}$ is a sequence of matrices such that $\dim(A_{n})=d_{n}\rightarrow\nf$, as $n\rightarrow\nf$ and with eigenvalues $\la_{j}(A_{n})$ and singular values $\si_{j}(A_{n})$, $j=1,\ldots,d_{n}$.
\begin{itemize}
\item We say that $\{A_{n}\}_{n}$ is {\em distributed as $f$ over $D$ in the sense of the eigenvalues,} and we write $\{A_{n}\}_{n}\sim_\la(f,D),$ if
\begin{equation}\label{distribution:eig}
\lim_{n\to\nf}\frac{1}{d_{n}}\sum_{j=1}^{d_{n}}F(\la_{j}(A_{n}))=
\frac1{\mu(D)} \int_{D} F(f(t))\d t,
\end{equation}
for every continuous function $F$ with compact support. In this case, we say that $f$ is the \emph{spectral symbol} of $\{A_{n}\}_{n}$.

\item We say that $\{A_{n}\}_{n}$ is {\em distributed as $f$ over $D$ in the sense of the singular values,} and we write $\{A_{n}\}_p\sim_\si(f,D),$ if
\begin{equation}\label{distribution:sv}
\lim_{n\to\nf}\frac{1}{d_{n}}\sum_{j=1}^{d_{n}}F(\si_{j}(A_{n}))=
\frac1{\mu(D)} \int_{D} F(|f(t)|)\d t,
\end{equation}
for every continuous function $F$ with compact support. In this case, we say that $f$ is the \emph{symbol} of $\{A_{n}\}_{n}$ in the sense of the singular values.
\end{itemize}
\end{definition}

Throughout the paper, when the domain can be easily inferred from the context, we replace the notation $\{A_{n}\}_{n}\sim_{\la,\si}(f,D)$ with $\{A_{n}\}_{n}\sim_{\la,\si} f$. A noteworthy result due to Tilli \cite{Ti98a} and Tyrtyshnikov and Zamarashkin \cite{TyZa98} is the following

\begin{theorem}\label{toeplitz distribution}
Let $f\in L^{1}([-\pi,\pi])$, then $\{T_{n}(f)\}_{{n}}\sim_\si(f,[-\pi, \pi]).$ If $f$ is a real-valued function almost everywhere,
then $\{T_{n}(f)\}_{{n}}\sim_\la(f,[-\pi, \pi]).$
\end{theorem}

In the sequel, we introduce the $\glt$ class, a $\ast$-algebra of matrix-sequences containing Toeplitz matrix-sequences. The formal definition of $\glt$ matrix-sequences is rather technical and can be found in the scalar unilevel, scalar multilevel, block unilevel, block multilevel in the following books and revue papers \cite{BaGa20b,BaGa20a,GaSe17,GaSe18}, respectively. The original construction is involved and needs a whole coherent set of definitions and mathematical objects; see \cite{Se03,Se06,Ti98b}. However, in the writing of the books and the reviews, the authors realized that the mathematical construction is equivalent to a set of operative axioms that can be used conveniently, in practice, for deciding if a given matrix-sequence is of $\glt$ type and for computing the related symbol. Therefore, we just give and briefly report and discuss four of these axioms of the $\glt$ class, which are sufficient for our purposes. The current formulation is taken from \cite{GaSe17}.

Throughout, we use the following notation
\[
\{A_{n}\}_{n}\glteq { \ka(x,\tht)},\qquad \ka\colon[0,1]\times[-\pi,\pi]\rightarrow\bC,
\]
to say that the sequence $\{A_{n}\}_{n}$ is $\glt$ sequence with $\glt$ symbol $\ka(x,\tht)$.

Here we list four main features of $\glt$ sequences.
\begin{itemize}
\item[{\bf GLT 1}] Let $\{A_{n}\}_{n}\glteq\ka$ with a function $\ka\colon G\rightarrow \bC$, $G=[0,1]\times[-\pi,\pi]$, then $\{A_{n}\}_{n}\sim_\si(\ka,G)$. If the matrices $A_{n}$ are Hermitian, then it also holds that $\{A_{n}\}_{n}\sim_\la(\ka,G)$.
\item[{\bf GLT 2}] The set of $\glt$ sequences forms a $\ast$-algebra, i.e., it is closed under linear combinations, products, conjugation, but also inversion when the symbol is invertible a.e. In formulae, let $\{ A_{n} \}_{n} \glteq \ka_{1}$ and $\{ B_{n} \}_{n} \glteq \ka_{2}$, then
\begin{itemize}
\item[$\bullet$] $\{\al A_{n} + \be B_{n}\}_{n} \glteq \al\ka_{1}+\be \ka_{2}, \quad \al, \be \in \bC;$
\item[$\bullet$] $\{A_{n}B_{n}\}_{n} \glteq \ka_{1} \ka_{2};$
\item[$\bullet$] $\{ A_{n}^{*} \}_{n} \glteq {\ka^*_{1}};$
\item[$\bullet$] $\{A^{-1}_{n}\}_{n} \glteq \ka_{1}^{-1}$ provided that $\ka_{1}$ is invertible a.e.
\end{itemize}
\item[{\bf GLT 3}] Any sequence of Toeplitz matrices $\{ T_{n}(f) \}_{n}$ generated by $f \in L^{1}([-\pi, \pi])$ is a $\glt$ matrix-sequence with symbol $\ka(x, \tht) = f(\tht)$. For a Riemann integrable function $a$ defined on $[0,1]$, the corresponding diagonal sampling matrix-sequence
\[
\{ D_{n}(a) \}_{n}
\]
is a $\glt$ matrix-sequence with symbol $\ka(x,\tht) = a(x)$ and entries given by $[D_{n}(a)]_{j,j}=a(\frac{j}{n})$, $j=1,\ldots,n$.
\item[{\bf GLT 4}] Let $\{A_{n}\}_{n}\sim_\si 0$. We say that $\{A_{n}\}_{n}$ is a \emph{zero-distributed matrix-sequence}. Every zero-distributed matrix-sequence is a $\glt$ sequence with symbol $0$ and viceversa, i.e., $\{A_{n}\}_{n}\sim_\si0$ $\iff$ $\{A_{n}\}_{n}\glteq 0$.
\end{itemize}

If $f$ is smooth enough, an informal interpretation of the limit relation \eqref{distribution:eig} (resp. \eqref{distribution:sv}) is that when $n$ is sufficiently large, then the eigenvalues (resp. singular values) of $A_{n}$ can be approximated by a sampling of $f$ (resp. $|f|$) on a uniform equispaced grid of the domain $D$, up to at most few outliers. Often this approximation is good enough: the notion of $\glt$ momentary symbols has been introduced for obtaining a more accurate approximation.

 \begin{definition}[$\glt$ momentary symbols] \label{def:momentarysymbols}
Let $\{X_{n}\}_{n}$ be a matrix-sequence and assume that there exist matrix-sequences $\{A_{n}^{(j)}\}_{n}$, scalar sequences $c_{n}^{(j)}$, $j=0,\ldots,t$,
and measurable functions $f_{j}$ defined over $[-\pi,\pi]\times [0,1]$, $t$ nonnegative integer independent of $n$, such that
\begin{eqnarray*}
&&\Big\{\frac{A_{n}^{(j)}}{ c_{n}^{(j)}}\Big\}_{n} \glteq f_{j}, \\[2ex]
&&c_{n}^{(0)}=1,\quad c_{n}^{(s)}=o(c_{n}^{(r)}),\quad t\ge s>r, \\
&&\{X_{n}\}_{n}=\{A_{n}^{(0)}\}_{n} + \sum_{j=1}^t \{A_{n}^{(j)}\}_{n}.
\end{eqnarray*}
Then,
\begin{equation*}
f_{n}=f_0+ \sum_{j=1}^t c_{n}^{(j)} f_{j}
\end{equation*}
is defined as the $\glt$ momentary symbol for $X_{n}$ and $\{f_{n}\}$ is the sequence of $\glt$ momentary symbols for the matrix-sequence $\{X_{n}\}_{n}$.

\end{definition}

Of course, in line with \cite{BaGa20b,BaGa20a}, the momentary symbols could be matrix-valued with a number of variables equal to $2d$ and domain $[-\pi,\pi]^d\times [0,1]^d$ if the basic matrix-sequences appearing in Definition \ref{def:momentarysymbols} are, up to proper scaling, matrix-valued and multilevel $\glt$ matrix-sequences.

Clearly there is an immediate link with the $\glt$ theory stated in the next result, but many other connections should be investigated.

\begin{theorem}
$\{X_{n}\}_{n}\glteq f_0$ and $
\lim_{n\to \nf} f_{n}=f_0$ uniformly on the definition domain.
\end{theorem}

Here by making reference to the approximations by Finite Differences, for a fixed positive integer $l$, we could consider operators of the form
\begin{equation}\label{mixed orders}
\sum_{s=0}^{l} (-1)^s\al_s\frac{\d ^{\,2q}}{\d x^{2q}}
\end{equation}
which, by linearity of the approximation technique of the involved operators, and by (\ref{f_q}), give raise to Toeplitz structures of the type in (\ref{eigs of superposition}) with the expression
\[
T_{n}\Big(\sum_{s=0}^l \al_s h^{2(l-s)} f_s\Big).
\]
In perfect analogy, in the case where the approximation is obtained via Isogeometric Analysis, for a fixed positive integer $l$, we find
\[
T_{n}\Big(\sum_{s=0}^l \al_s h^{2(l-s)} g_s\Big),
\]
taking into account of (\ref{g_q}) and again the linearity of the approximation technique and of the considered operators. In both cases it is evident that the related matrix-sequences have $\sum_{s=0}^l \al_s h^{2(l-s)} F_s$ as $\glt$ momentary symbols in the sense of Definition \ref{def:momentarysymbols}, with $F_s$ being either $f_s$ or $g_s$, $s=0,1,\ldots,l$. In both cases we are interested to use the related $\glt$ momentary symbols in terms of superposition effect for computing in a fast way the eigenvalues
\begin{equation}\label{eigs of superposition example}
\la_{j}\Big(T_{n}\Big(\sum_{s=0}^l \al_s h^{2(l-s)} F_s\Big)\Big).
\end{equation}
Notice again that (\ref{eigs of superposition example}) is a special instance of the general problem indicated in (\ref{eigs of superposition}).

\section{Main results}\label{sec:Main}

For a constant $\al\ge0$, the well-known weighted Wiener algebra $W^{\al}$ is the collection of all functions $f\colon[0,2\pi]\to\bC$ whose Fourier coefficients $\hat f_{j}$ satisfy
\[\|f\|_{\al}\coloneqq\sum_{j=-\nf}^{\nf} |\hat f_{j}|(|j|+1)^{\al}<\nf.\]
We address real-valued symbols $f$ in $W^{\al}$, tracing out a simple loop and satisfying the following conditions:

\begin{enumerate}[(i)]
\item The range of $f$ is a segment $[0,\mu]$ with $\mu>0$.
\item $f(0)=f(2\pi)=0$, $f'(0)=f'(2\pi)=0$, and $f''(0)=f''(2\pi)>0$.
\item There is a $\si_0\in(0,2\pi)$ such that $f(\si_0)=\mu$, $f'(\si)>0$ for $0<\si<\si_0$, $f'(\si)<0$ for $\si_0<\si<2\pi$, $f'(\si_0)=0$, and $f''(\si_0)<0$.
\end{enumerate}

The collection of all these symbols is called the \textit{simple loop class} and is denoted by $\SL^{\al}$.
In this note we consider symmetric symbols in $\SL^{\al}$, that is $\si_0=\pi$ and $f(s)=f(2\pi-s)$ for each $s\in[0,\pi]$. Obviously, it is enough to study such a symbols in the interval $[0,\pi]$.

For every $\la\in[0,\mu]$ there exists a unique $s\in[0,\pi]$ satisfying $f(s)=\la$, and the symbol $f-\la$ has 2 zeros: $\pm s$, implying that the Toeplitz operator $T(f-\la)$ is not invertible. According to the simple-loop method (see~\cite{BoGr17}), by considering
\begin{eqnarray}\label{eq:b}
b_{f}(\si,s)&\coloneqq&\frac{f(\si)-f(s)}{4\sin\big(\mfrac{\si-s}{2}\big)\sin\big(\mfrac{\si+s}{2}\big)}\notag\\
&=&\frac{f(\si)-f(s)}{2(\cos(s)-\cos(\si))}\qquad(\si\in[0,2\pi],\ s\in[0,\pi]),
\end{eqnarray}
we obtain a real and continuous function, which is also bounded away from zero. The resulting operator $T(b_{f}(\cdot,s))$ is invertible and therefore, since the finite section method can be applied (see~\cite{BoSi99} for example), the related finite Toeplitz matrices $T_{n}(b_{f}(\cdot,s))$ are also invertible. Note that $b_{f}$ can be thought of as the quotient between $f-\la$ and $4\sin(\frac{\si-s}{2})\sin(\frac{\si+s}{2})$, which is similar to the preconditioning process of the ill-conditioned matrix $T_{n}(f-\la)$ used for example in \cite{DiBFi93,Se95}; for a general account on preconditioning in a Toeplitz setting see \cite{Ng04,GaSe17} and references therein.

For a function $u$ with a singularity at some point in the interval $I$, let $\dint_{I} u(x)\d x$ be the Cauchy principal value of the singular integral $\int_{I}u(x)\d x$. The function $b_{f}$ admits the so called Wiener--Hopf factorization $b_{f}=[b_{f}]_-[b_{f}]_+$ (index zero) with
\[[b_{f}]_{\pm}(t,s)\coloneqq \exp\Big\{\frac{1}{2}\log b_{f}(t,s)\pm\frac{1}{2\pi\i}\dint_{\bT}\frac{\log b_{f}(\tau,s)}{\tau-t}\d \tau\Big\},\]
where $t=\e^{\i\si}$ and $\si\in[0,2\pi]$.
We recall that the Wiener--Hopf factorization (also called method or decomposition) was introduced by N. Wiener and E. Hopf in 1931. What we call Wiener--Hopf factorization has its origin in the work of Gakhov \cite{Ga37}, but Mark Krein \cite{Kr58} was the first to understand the operator theoretic essence and its algebraic background, and to present it in a clear way; for a nice and modern explanation see \cite[\S1.4]{BoGr05}.

As in~\cite{BoBo15a, BoGr17}, we define the function $\eta_{f}\colon[0,\pi]\to\bR$ by
\begin{eqnarray}\label{eq:eta}
\eta_{f}(s)&\coloneqq&\frac{1}{4\pi}\dint_0^{2\pi}\frac{\log b_{f}(\si,s)}{\tan\big(\mfrac{\si-s}{2}\big)}\d \si-\frac{1}{4\pi}\dint_0^{2\pi}\frac{\log b_{f}(\si,s)}{\tan\big(\mfrac{\si+s}{2}\big)}\d \si\notag\\
&=&\frac{\sin(s)}{2\pi}\dint_0^{2\pi}\frac{\log b_{f}(\si,s)}{\cos(s)-\cos(\si)}\d \si.
\end{eqnarray}
Now \cite[Theorem 2.3]{BoGr17} tells us, in particular, that the eigenvalues of $T_{n}(f)$ are given by
\begin{equation}\label{eq:lamain}
\la_{j}(T_{n}(f))=f(d_{j,n})+\sum_{\ell=1}^{\lfloor\al\rfloor}c_{\ell}(d_{j,n})\,h^{\ell}+E_{j,n,\al},
\end{equation}
where
\begin{itemize}
\item the eigenvalues of $T_{n}(f)$ are arranged in nondecreasing order;
\item $h\coloneqq\frac{1}{n+1}$ and $d_{j,n}\coloneqq\pi j h$;
\item the coefficients $c_{\ell}$ depend only on $f$ and can be found explicitly, for example
\begin{equation}\label{eq:c12}
c_{1}=-f'\eta_{f},\qquad c_{2}=\frac{1}{2}f''\eta_{f}^{2}+f'\eta_{f}\eta_{f}';
\end{equation}
\item $E_{j,n,\al}=O(h^{\al})$ is the remainder (error) term, which satisfies the bounding $|E_{j,n,\al}|\le \ka_{\al}h^{\al}$ for some constant $\ka_{\al}$ depending only on $\al$ and $f$.
\end{itemize}

For $f,g\in\SL^{\al}$ and a constant $\be\in\bR$, we investigate the relationship between the eigenvalues $\la_{j}(T_{n}(f))$, $\la_{j}(T_{n}(g))$, and $\la_{j}(T_{n}(f+\be g))$. From \eqref{eq:lamain} we easily obtain
\[\la_{j}(T_{n}(f+\be g))=f(d_{j,n})+\be g(d_{j,n})+O(h).\]
Indeed, as a challenge in the field, we are looking for a more detailed result involving a complete expansion and a real constant $\be_{n}$ depending on $n$.
The symbol $f+\be_{n} g$ depends on $n$, as a consequence the actual simple-loop method can not be applied. However, under proper adjustments, the quoted technique can be used when $\be_{n}$ is $h$ or $h^h$, see Theorems \ref{th:SLh} and \ref{th:SLhh}.

We start with the symbol $f+gh$, that is $\be_{n}=h$. Consider the function
\[\psi(s)\coloneqq\frac{\sin(s)}{2\pi}\dint_0^{2\pi}\frac{b_{g}(\si,s)}{b_{f}(\si,s)(\cos(s)-\cos(\si))}\d \si.\]
\begin{theorem}\label{th:SLh}
Let $f$ and $g$ be two symmetric symbols in $\SL^{\al}$ with $\al\ge2$. Then
\[\la_{j}(T_{n}(f+g h))=f(d_{j,n})+\sum_{\ell=1}^{\lfloor\al\rfloor}\Psi_{\ell}(d_{j,n}) h^{\ell}+E_{j,n,\al},\]
where $\Psi_{\ell}$ are bounded functions from $[0,\pi]$ to $\bR$ depending only on $f,g$ that can be expressed explicitly. For instance
\begin{eqnarray*}
\Psi_{1}&=&g-f'\eta_{f},\\
\Psi_{2}&=&\frac{1}{2}f''\eta_{f}^{2}+f'\eta_{f}\eta_{f}'-f'\psi-g'\eta_{f}.
\end{eqnarray*}
The remainder $E_{j,n,\al}=O(h^{\al})$ satisfies the bounding $|E_{j,n,\al}|\le \ka_{\al}h^{\al}$ for some constant $\ka_{\al}$ depending only on $\al$, $f$, and $g$.
\end{theorem}

For the next result we will study the eigenvalues corresponding to the symbol $f+g h^h$, that is $\be_{n}=h^h$. Consider the function
\[\ph(s)\coloneqq\frac{\sin(s)}{2\pi}\dint_0^{2\pi}\frac{b_{g}(\si,s)}{(b_{f}(\si,s)+b_{g}(\si,s))(\cos(s)-\cos(\si))}\d \si.\]
\begin{theorem}\label{th:SLhh}
Let $f$ and $g$ be two symmetric symbols in $\SL^{\al}$ with $\al\ge2$. Then
\[\la_{j}(T_{n}(f+g h^h))=f(d_{j,n})+g(d_{j,n})+\sum_{\ell=1}^{\lfloor\al\rfloor}\sum_{k=0}^{\ell}\Ga_{\ell,k}(d_{j,n})h^{\ell}\log^k(h)+E_{j,n,\al},\]
where $\Ga_{\ell,k}$ are bounded functions from $[0,\pi]$ to $\bR$ depending only on $f, g$ that can be expressed explicitly. For instance
\begin{eqnarray*}
\Ga_{1,1}&\coloneqq&g,\\
\Ga_{1,0}&\coloneqq&-\eta_{f+g}(f'+g'),\\
\Ga_{2,2}&\coloneqq&\frac{1}{2}g,\\
\Ga_{2,1}&\coloneqq&-\ph(f'+g')-g'\eta_{f+g},\\
\Ga_{2,0}&\coloneqq&\frac{1}{2}\eta_{f+g}^{2}(f''+g'')+\eta_{f+g}\eta_{f+g}'(f'+g').
\end{eqnarray*}
The remainder (error) term $E_{j,n,\al}=O(h^{\al}|\log^{\al}(h)|)$ satisfies $|E_{j,n,\al}|\le \ka_{\al}h^{\al}|\log^{\al}(h)|$ for some constant $\ka_{\al}$ depending only on $\al$, $f$, and $g$.
\end{theorem}

\section{Proof of the main results}\label{sec:Main-Proof}

Our specific aim is to recreate the so called simple-loop method working with the symbol $f+\be_{n} g$ for two particular values of $\be_{n}$, i.e. $h$ and $h^h$. As we will show, in both cases we arrive at the important equation
\begin{equation}\label{eq:Main}
\frac{s}{h}+\eta_{f+\be_{n} g}(s)=\pi j+\De_{j,n,\al},
\end{equation}
where $\De_{j,n,\al}$ satisfies some smoothness and bounding conditions (the previous relation \eqref{eq:Main} is the key for the calculation of asymptotic expansions for the eigenvalues of several Toeplitz matrix-sequences, see for example \cite[Eq.\,4.15]{BoBo15a} and \cite[Eq.\,5.7]{BoGr17}). At this point we decided to expand $\eta_{f+\be_{n} g}$ into factors with coefficients not involving $n$, then we made a technical work claiming the function on the left side of \eqref{eq:Main} is a contraction and henceforth, we can use the Banach fixed point theorem in order to solve it for $s$. Thus, we start with the following technical result.

\begin{lemma}\label{lm:eta}
As $n\to\nf$, we have
\begin{enumerate}
\item[\textup{(i)}] $\eta_{f+gh}(s)=\eta_f(s)+\psi(s)h+O(h^{2})$,
\item[\textup{(ii)}] $\eta_{f+gh^h}(s)=\eta_{f+g}(s)+\ph(s) h\log(h)+O(h^{2}\log^{2}(h))$,
\end{enumerate}
where $\psi,\ph$ are $\lfloor\al\rfloor-1$ times continuously differentiable functions, given by
\begin{eqnarray*}
\psi(s)&\coloneqq&\frac{\sin(s)}{2\pi}\dint_0^{2\pi}\frac{b_{g}(\si,s)}{b_{f}(\si,s)(\cos(s)-\cos(\si))}\d \si,\\
\ph(s)&\coloneqq&\frac{\sin(s)}{2\pi}\dint_0^{2\pi}\frac{b_{g}(\si,s)}{(b_{f}(\si,s)-b_{g}(\si,s))(\cos(s)-\cos(\si))}\d \si,
\end{eqnarray*}
\end{lemma}

\begin{proof}
By \cite[Lemma 4.4]{BoGr17} we know that $s\mapsto b_{f}(\cdot,s)$ and $s\mapsto b_{g}(\cdot,s)$ act continuously from $[0,\pi]$ to $W^{\al-1}$, which combined with the fact that the singular integral operator is continuous over the weighted Wiener algebras, claims that $\psi,\ph$ are $\lfloor\al\rfloor-1$ times continuously differentiable functions. For the part (i), we use the expansion
\begin{eqnarray*}
\log(b_{f}+b_{g} h)&=&\log(b_{f})+\log'(b_{f})b_{g}h+\frac{1}{2}\log''(b_{f})b_{g}^{2}h^{2}+O(h^{3})\\
&=&\log(b_{f})+\frac{b_{g}}{b_{f}}h-\frac{b_{g}^{2}}{2b_{f}^{2}}h^{2}+O(h^{3}).
\end{eqnarray*}
From the definition of $\eta$ \eqref{eq:eta} we get
\begin{eqnarray*}
\eta_{f+gh}(s)-\eta_{f}(s)&=&\frac{\sin(s)}{2\pi}\dint_0^{2\pi} \frac{\log(b_{f}(\si,s)+b_{g} h(\si,s))-\log b_{f}(\si,s)}{\cos(s)-\cos(\si)}\d \si\\
&=&\psi(s) h+O(h^{2}),
\end{eqnarray*}
proving the first part.

For the part (ii), we use the expansion
\[h^h=\e^{h\log(h)}=1+h\log(h)+\frac{1}{2}h^{2}\log^{2}(h)+O(h^{3}|\log^{3}(h)|)\]
and we easily arrive at
\[b_{f}+b_{g} h^h=b_{f}+b_{g}+b_{g} h\log(h)+\frac{b_{g}}{2}h^{2}\log^{2}(h)+O(h^{3}|\log^{3}(h)|).\]
Note that for every $n$ the symbols $b_{f}+b_{g}h^h$ and $b_{f}+b_{g}$ are positive, bounded, and bounded away from zero. Thus we can apply the logarithm to the former and expand around the latter so that
\[\log(b_{f}+b_{g}h^h)=\log(b_{f}+b_{g})+\frac{b_{g}}{b_{f}+b_{g}}h\log(h)+\frac{b_{f}b_{g}}{2(b_{f}+b_{g})}h^{2}\log^{2}(h)+O(h^{3}|\log^{3}(h)|).\]
Finally, combining this expansion with \eqref{eq:eta}, the second part is proven.
\end{proof}

The following lemma can be derived directly from the expansion \eqref{eq:lamain}. It shows how the eigenvalues of the sum of two Toeplitz matrices are related with the individual ones, which is highly nontrivial given the inherent nonlinearity of the eigenvalues with respect to the entries of the considered matrix.

\begin{lemma}\label{lm:addition}
Let $f$ and $g$ be two symmetric symbols in $\SL^{\al}$ with $\al\ge2$. Then we have
\begin{eqnarray*}
\la_{j}(T_{n}(f+g))&=&\la_{j}(T_{n}(f))+\la_{j}(T_{n}(g))+\sum_{\ell=1}^{\lfloor\al\rfloor} Q_{\ell}(d_{j,n})h^{\ell}+E_{j,n,\al},
\end{eqnarray*}
where $Q_{\ell}$ are bounded functions from $[0,\pi]$ to $\bR$ depending only on $f, g$, that can be expressed explicitly. For example
\begin{eqnarray*}
Q_{1}&\coloneqq&f'(\eta_{f}-\eta_{f+g})+g'(\eta_{g}-\eta_{f+g}),\\
Q_{2}&\coloneqq&\frac{f''}{2}(\eta_{f+g}^{2}-\eta_{f}^{2})+\frac{g''}{2}(\eta_{f+g}^{2}-\eta_{g}^{2})\\
&&+f'(\eta_{f+g}\eta_{f+g}'-\eta_{f}\eta_{f}')+g'(\eta_{f+g}\eta_{f+g}'-\eta_{g}\eta_{g}').
\end{eqnarray*}
The remainder $E_{j,n,\al}=O(h^{\al})$ satisfies the bounding $|E_{j,n,\al}|\le \ka_{\al}h^{\al}$ for some constant $\ka_{\al}$ depending only on $\al$, $f$, and $g$.
\end{lemma}

\begin{proof}
Let $c_k^{f}$ $(k=1,2)$ be the function involved in the expansion \eqref{eq:lamain} associated with the symbol $f$. We recall that $\eta_{f+g}$ is associated with $f+g$. It is easy to show that $b=b_{f}+b_{g}$. The situation with $\eta_{f+g}$ is more delicate. Consider the function $\rho\coloneqq\eta_{f+g}-\eta_{f}-\eta_{g}$. From \eqref{eq:c12} we obtain
\begin{eqnarray*}
c_{1}^{f+g}&=&-f'\eta_{f+g}\\
&=&-(f'+g')(\eta_{f}+\eta_{g}+\rho)\\
&=&c_{1}^{f}+c_{1}^{g}+Q_{1},
\end{eqnarray*}
where $Q_{1}\coloneqq f'(\eta_{f}-\eta_{f+g})+g'(\eta_{g}-\eta_{f+g})$. Similarly
\begin{eqnarray*}
c_{2}^{f+g}&=&\frac{1}{2}f''\eta_{f+g}^{2}+f'\eta_{f+g}\eta_{f+g}'\\
&=&\frac{1}{2}(f''+g'')(\eta_{f}+\eta_{g}+\rho)^{2}+(f'+g')(\eta_{f}+\eta_{g}+\rho)(\eta_{f}'+\eta_{g}'+\rho')\\
&=&c_{2}^{f}+c_{2}^{g}+Q_{2},
\end{eqnarray*}
where
\begin{eqnarray*}
Q_{2}&\coloneqq&\frac{1}{2}f''(\eta_{f+g}^{2}-\eta_{f}^{2})+\frac{1}{2}g''(\eta_{f+g}^{2}-\eta_{g}^{2})\\
&&+f'(\eta_{f+g}\eta_{f+g}'-\eta_{f}\eta_{f}')+g'(\eta_{f+g}\eta_{f+g}'-\eta_{g}\eta_{g}').
\end{eqnarray*}
According to the simple-loop method, all the involved functions in the terms $Q_{\ell}$ are bounded (see \cite[\S 3]{BoBo15a}), thus they are also bounded. Finally, putting together the previous expressions for $c_{1}^{f+g}$ and $c_{2}^{f+g}$ in \eqref{eq:lamain}, the proof is over.
\end{proof}

For the proofs of the main theorems we need to mimic the complete calculation of the simple-loop method, but since most of the modifications are very simple, we decided to mention only the most relevant. From \eqref{eq:b} it is clear that $b_{f+\be_{n} g}=b_{f}+\be_{n} b_{g}$. Following \cite{BoGr17} we can see that all the results in Section 3 can be applied to the symbol $f+\be_{n} g$ for any real and positive $\be_{n}$ bounded uniformly for $n\in\bN$, and that all the results in Section 4 can be applied to $b_{f}$ and $b_{g}$, separately. One of the key ingredients of Section 5 is the function $\Tht_k(t,s)\coloneqq[T_k^{-1}(b(\cdot,s))\chi_0](t)$, where $\chi_0(t)\coloneqq 1$, who is a polynomial satisfying certain properties. After the adjustments, in \cite[Lemma 5.1]{BoGr17} we obtained
\[\Tht_k(t,s)=[\{T_k(b_{f}(\cdot,s))+\be_{n} T_k(b_{g}(\cdot,s))\}^{-1}\chi_0](t),\]
which is still a polynomial with the same properties, but its coefficients now depend on $n$.

After the implementation of the following minor changes, we was able to replicate the result in Lemma 5.4, even with the same remaining function $R_{1}^{(n)}$.
\begin{itemize}
\item To find a bound for the operator $\|T_{n}^{-1}(b_{f+\be_{n} g}(\cdot,s))\|$ in the Lebesgue space $\ell_{1}$, we proceed as follows, since the function $b_{f+\be_{n} g}=b_{f}+\be_{n} b_{g}$ is positive, bounded, and bounded away from zero, the Toeplitz matrix $T_{n}(b_{f+\be_{n} g})$ is positive definite and invertible, hence its operator norm equals its largest eigenvalue (which must be real and positive), and which is naturally bounded from below by $\inf\{b_{f+\be_{n} g}(s)\colon 0\le s\le\pi\}$. Then
\[\|T_{n}^{-1}(b_{f+\be_{n} g}(\cdot,s))\|\le\inf\{b_{f}(s)+\be_{n} b_{g}(s)\colon 0\le s\le\pi\}^{-1}.\]
If the constant $\be_{n}$ takes the value $h$, the last infimum is greater than $\inf\{b_{f}(s)\colon 0\le s\le\pi\}$, while for $\be_{n}=h^h$ it is greater than
\[\inf\{b_{f}(s)\colon 0\le s\le\pi\}+\inf\{b_{g}(s)\colon 0\le s\le\pi\}.\]
In both cases the provided bound is positive, finite, and does not depend on $n$.
\item For the two values of $\be_{n}$ considered, the function $b_{f+\be_{n} g}$ is continuous, bounded, and bounded away from zero uniformly in $n$. Thus, the boundness of the Wiener--Hopf factor $[b_{f+\be_{n} g}]_+^{-1}$ is a direct consequence of the fact that the singular integral operator is bounded over the space of continuous functions.
\end{itemize}

Lemma \ref{lm:55} is a plain and direct consequence, which for the reader convenience we present here, but without a proof.
\newpage
\begin{lemma}\label{lm:55}
For every sufficiently large natural number $n$ there exists a real-valued function $R_{2}^{(n)}\in C[0,\pi]$ with the following properties:
\begin{enumerate}
\item[\textup{(i)}] a number $\la=f(s)+\be_{n} g(s)$ is an eigenvalue of $T_{n}(f+\be_{n} g)$ if and only if there is a $j\in\bZ$ such that
\[\frac{s}{h}+\eta_{f+\be_{n} g}(s)-R_{2}^{(n)}(s)=j\pi;\]
\item[\textup{(ii)}] $R_{2}^{(n)}(0)=R_{2}^{(n)}(\pi)=0$;
\item[\textup{(iii)}] $\|R_{2}^{(n)}\|_{L_\nf}=o(h^{\al-1})$ as $n\to\nf$.
\end{enumerate}
\end{lemma}

\begin{proof}[Proof of Theorem \ref{th:SLh}]
Assume that $2\le\al<3$ and let $\be_{n}=h$ and for simplicity write $d$ instead of $d_{j,n}$. Combining Lemmas \ref{lm:eta} and \ref{lm:55} part (iii), we obtain
\begin{equation}\label{eq:etahMain}
G_{n}(s)=j\pi,
\end{equation}
where $G_{n}(s)\coloneqq \frac{s}{h}+\eta_{f}(s)+\psi(s)h-R_{2}^{(n)}(s)$. The function $G_{n}$ is continuous on the interval $[0,\pi]$ with $G_{n}(0)=0$ and $G_{n}(\pi)=\frac{\pi}{h}$, hence by \cite[Theorem 2.1]{BoGr17} we know that
\begin{enumerate}
\item[\textup{(i)}] the eigenvalues of $T_{n}(f+gh)$ are all distinct:
\[\la_{1}(T_{n}(f+gh))<\cdots<\la_{n}(T_{n}(f+gh));\]
\item[\textup{(ii)}] the numbers $s_{j,n}\coloneqq \big[f+gh\big]_{[0,\pi]}^{-1}(\la_{j}(T_{n}(f+gh))$ $(j=1,\ldots,n)$ satisfy the relation \eqref{eq:etahMain} with $R_{2}^{(n)}(s_{j,n})=o(h^{\al-1})$ as $n\to\nf$, uniformly in $j$;
\item[\textup{(iii)}] for every sufficiently large $n$, \eqref{eq:etahMain} has exactly one solution $s_{j,n}\in[0,\pi]$ for each $j=1,\ldots,n$.
\end{enumerate}

Let $F_{n}(s)\coloneqq \frac{s}{h}+\eta_{f}(s)+\psi(s)h$. Hence, Lemmas 6.1 and 6.2 in \cite{BoGr17} can be replicated asserting that $F_{n}(s)=\pi j$ has a unique solution $\hat s_{j,n}$ for each $j=1,\ldots,n$ satisfying the bounding $|s_{j,n}-\hat s_{j,n}|=o(h^{\al})$, and that the function
\[\Phi_{j,n}(s)\coloneqq d-\eta_{f}(s)h-\psi(s)h^{2},\]
is a contraction on $[0,\pi]$. Then by the Banach fixed point theorem, the sequence defined by
\[\hat s_{j,n}^{(0)}\coloneqq d\quad\mbox{and}\quad\hat s_{j,n}^{(\ell)}\coloneqq\Phi_{j,n}(\hat s_{j,n}^{(\ell-1)})\quad(\ell\ge1),\]
satisfies $|\hat s_{j,n}-\hat s_{j,n}^{(\ell)}|=O(h^{\ell+1})$.

Write \eqref{eq:etahMain} as $s=\Phi_{j,n}(s)+\De_{n,\al}(s)$ where $\De_{n,\al}(s)=o(h^{\al})+O(h^{3})$ as $n$ grows to infinity. We will iterate over the relation $s=\Phi_{j,n}(s)$. We have $\hat s_{j,n}^{(0)}=d$ and $\hat s_{j,n}^{(1)}=\Phi_{j,n}(\hat s_{j,n}^{(0)})=d-\eta_{f}(d)h-\psi(d)h^{2}$. To evaluate $\hat s_{j,n}^{(2)}$ first note that

\begin{eqnarray*}
\eta_{f}(\hat s_{j,n}^{(1)})&=&\eta_{f}(d)-\eta_{f}'(d)\{\eta_{f}(d)h+\psi(d)h^{2}\}\\
&&+\frac{1}{2}\eta_{f}''(d)\{\eta_{f}^{2}(d)+\psi^{2}(d)h^{2}+2\eta_{f}(d)\psi(d)h\}h^{2}+O(h^{3})\\
&=&\eta_{f}(d)-\eta_{f}(d)\eta_{f}'(d)h+\big\{\frac{1}{2}\eta_{f}^{2}(d)\eta_{f}''(d)-\eta_{f}'(d)\psi(d)\big\}h^{2}+O(h^{3})
\end{eqnarray*}
and that
\begin{eqnarray*}
\psi(\hat s_{j,n}^{(1)})&=&\psi(d)-\psi'(d)\{\eta_{f}(d)h+\psi(d)h^{2}\}\\
&&+\frac{1}{2}\psi''(d)\{\eta_{f}^{2}(d)+\psi^{2}(d)h^{2}+2\eta_{f}(d)\psi(d)h\}h^{2}+O(h^{3})\\
&=&\psi(d)-\psi'(d)\eta_{f}(d)h+O(h^{2}).
\end{eqnarray*}
Now by the Taylor theorem we obtain
\begin{eqnarray*}
\hat s_{j,n}^{(2)}&=&d-\eta_{f}(d-\eta_{f}(d)h-\psi(d)h^{2})h-\psi(d-\eta_{f}(d)h-\psi(d)h)h^{2}\\
&=&d-\eta_{f}(d)h+\{\eta_{f}(d)\eta_{f}'(d)-\psi(d)\}h^{2}+O(h^{3}),
\end{eqnarray*}
which combined with $|s_{j,n}-\hat s_{j,n}^{(2)}|=o(h^{\al})+O(h^{3})$ leads to prove the theorem for the case $2\le\al<3$. The remaining cases can be readily proved, essentially in the same manner.
\end{proof}

\begin{proof}[Proof of Theorem \ref{th:SLhh}]
This proof is quite similar to the previous one, thus to avoid unnecessary repetition, we will only remark the main differences. Taking $\be_{n}=h^h$ we obtain
\begin{eqnarray*}
G_{n}(s)&\coloneqq&\frac{s}{h}+\eta_{f+g}(s)+\ph(s)h\log(h),\\
F_{n}(s)&\coloneqq&\frac{s}{h}+\eta_{f+g}(s)+\ph(s)h\log(h),\\
\Ph_{j,n}(s)&\coloneqq&d-\eta_{f+g}(s)h-\ph(s)h^{2}\log(h).
\end{eqnarray*}
The remaining function $\De_{n,\al}$ now have order $o(h^{\al})+O(h^{3}\log^{2}(h))$ as $n\to\nf$, and $\hat s_{j,n}^{(0)}=d$ and $\hat s_{j,n}^{(1)}=\Phi_{j,n}(\hat s_{j,n}^{(0)})=d-\eta_{f+g}(d)h-\ph(d)h^{2}\log(h)$. To evaluate $\hat s_{j,n}^{(2)}$ first note that
\begin{eqnarray*}
\eta_{f+g}(\hat s_{j,n}^{(1)})&=&\eta_{f+g}(d)-\eta_{f+g}'(d)\{\eta_{f+g}(d)h+\ph(d)h^{2}\log(h)\}\\
&&+\frac{1}{2}\eta_{f+g}''(d)\{\eta_{f+g}^{2}(d)+\ph^{2}(d)h^{2}\log^{2}(h)
\\&&+2\eta_{f+g}(d)\ph(d)h\log(h)\}h^{2}+O(h^{3}|\log^{3}(h)|)\\
&=&\eta_{f+g}(d)-\eta_{f+g}(d)\eta_{f+g}'(d)h\\
&&+\big\{\frac{1}{2}\eta_{f+g}^{2}(d)\eta_{f+g}''(d)-\eta_{f+g}'(d)\ph(d)\log(h)\big\}h^{2}+O(h^{3}|\log^{3}(h)|).
\end{eqnarray*}
and that
\begin{eqnarray*}
\ph(\hat s_{j,n}^{(1)})&=&\ph(d)-\ph'(d)\{\eta_{f+g}(d)h+\ph(d)h^{2}\log(h)\}\\
&&+\frac{1}{2}\ph''(d)\{\eta_{f+g}^{2}(d)+\ph^{2}(d)h^{2}\log^{2}(h)+2\eta_{f+g}(d)\ph(d)h\log(h)\}h^{2}+O(h^{3}|\log^{3}(h)|)\\
&=&\ph(d)-\ph'(d)\eta_{f+g}(d)h+O(h^{2}|\log(h)|).
\end{eqnarray*}
Now by the Taylor theorem we obtain
\[\hat s_{j,n}^{(2)}=d-\eta_{f+g}(d)h+\{\eta_{f+g}(d)\eta_{f+g}'(d)-\ph(d)\log(h)\}h^{2}+O(h^{3}|\log^{3}(h)|),\]
which combined with $|s_{j,n}-\hat s_{j,n}^{(2)}|=o(h^{\al})+O(h^{3}|\log^{3}(h)|)$ give us the theorem for the case $2\le\al<3$. The remaining cases can be readily proved.
\end{proof}

\section{Numerical tests and an algorithm proposal}\label{sec:Num}

The purpose of our main results is to provide asymptotic expansions revealing the fine structure of the eigenvalue behavior in a matrix-less fashion, that is, without the need of any matrix inversion, multiplication by vectors, or even the storage of its entries. Such an expansion has proved to be useful in several applications. In \cite{EkGa19,EkGa18} the authors exploited those expansions and provided a clever and fully numerical algorithm, that calculates the eigenvalues of arbitrarily large matrices within machine precision. More specifically they devised an extrapolation algorithm for computing the eigenvalues of banded symmetric Toeplitz matrices with a high level of accuracy and low computation cost, starting from the computation of the eigenvalues of small size matrices, in the same spirit of the classical extrapolation procedures for the summation of smooth functions.

As in previous works, the purpose of this section is to check if we calculate the involved coefficients in all of our results correctly, and we want to prove that our asymptotic expressions deliver good approximations even for values of $n$ in the early hundreds. The numerical calculation of a singular integral can be very difficult, we therefore use a standard regularization trick (see~\cite[\S7]{BoBo15a}) which works as follows. For an integral of the form
\[\cI(v)=\dint_{\ga}^{\de}\frac{f(u,v)}{h(u,v)}\d u,\]
where for some $u_0\in[\ga,\de]$, we have $h(u_0,v)=0$ and $f(u_0,v)\ne0$, we can write
\[\cI(v)=\dint_{\ga}^{\de}\frac{f(u,v)-f(u_0,v)}{h(u,v)}\d u+f(u_0,v)\dint_{\ga}^{\de}\frac{\d u}{h(u,v)},\]
and in most cases, the above integrals are easier to tackle than $\cI$ itself. For the case of $\eta$, the value of $\dint_0^{2\pi}\cot\big(\frac{\si-s}{2}\big)\d \si=0$, and hence we obtain
\[\eta(s)=\frac{\sin(s)}{2\pi}\dint_0^{2\pi}\frac{\log b(\si,s)-\log b(s,s)}{\cos(s)-\cos(\si)}\d \si,\]
where $b(s,s)$ can be calculated with the L'Hôpital rule as
\[b(s,s)=\frac{f'(s)}{2\sin(s)},\quad b(0,0)=\frac{f''(0)}{2},
\quad\mbox{and}\quad b(\pi,\pi)=-\frac{f''(0)}{2}.\]

Consider the simple-loop symbol given by
\[f(\si)\coloneqq \frac{(1+\rho)^{2}}{2}\cdot\frac{1-\cos(\si)}{1-2\rho\cos(\si)+\rho^{2}}\qquad(0\le\si\le2\pi),\]
for a constant $0<\rho<1$. The respective Fourier coefficients can be explicitly calculated as $\hat f_k=\frac{1}{4}(\rho^{2}-1)\rho^{|k|-1}$ for $k\ne0$ and $\frac{1}{2}(1+\rho)$ for $k=0$. This symbol was inspired in the Kac--Murdock--Szeg\H{o} Toeplitz matrices introduced in \cite{KaMu53} and subsequently studied in \cite{Tr99,Tr10}, which usually are present in important physics models. We have
\[\|f\|_{\al}=\frac{1+\rho}{2}+\frac{\rho^{2}-1}{2\rho}\sum_{k=1}^{\nf}\rho^k(k+1)^{\al},\]
which is finite for every $\al>0$, the remaining simple-loop conditions are easily verified in Figure~\ref{fg:f0}. Then $f\in\SL^{\al}$ for any $\al>0$. In this case it is possible to deduce elementary expressions for the factors of the Wiener--Hopf factorization, indeed
\[[b_{f}]_+(t,s)=\frac{(1-\rho^{2})^{2}}{4(1-\rho t)(1-2\rho\cos(s)+\rho^{2})},\qquad
[b_{f}]_-(t,s)=\frac{1}{1-\rho t^{-1}}.\]
The function $\eta_{f}$ from \eqref{eq:eta} is therefore nicely given by
\[\eta_{f}(s)=2\arctan\Big(\frac{\rho\sin(s)}{1-\rho\cos(s)}\Big).\]

\begin{figure}[ht]
\centering
\includegraphics[width=60mm]{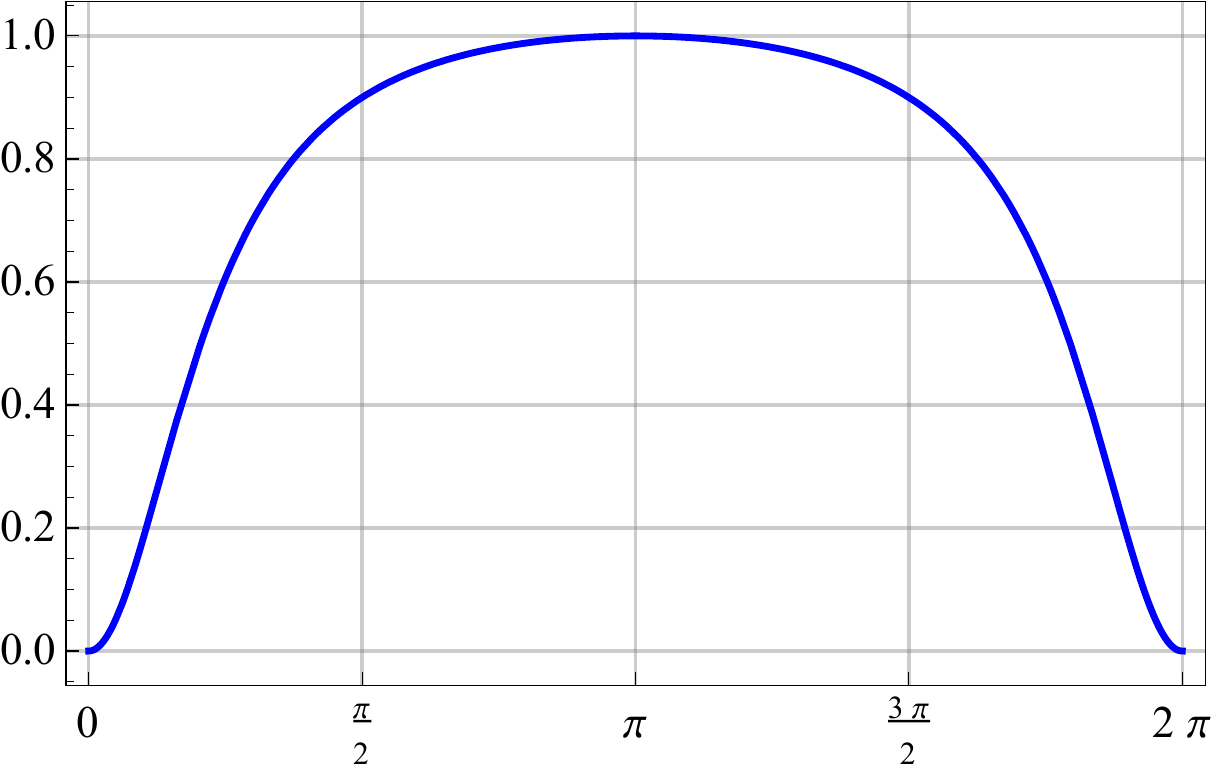}
\put(-84,113){\small $f$}
\hspace{2mm}
\includegraphics[width=60mm]{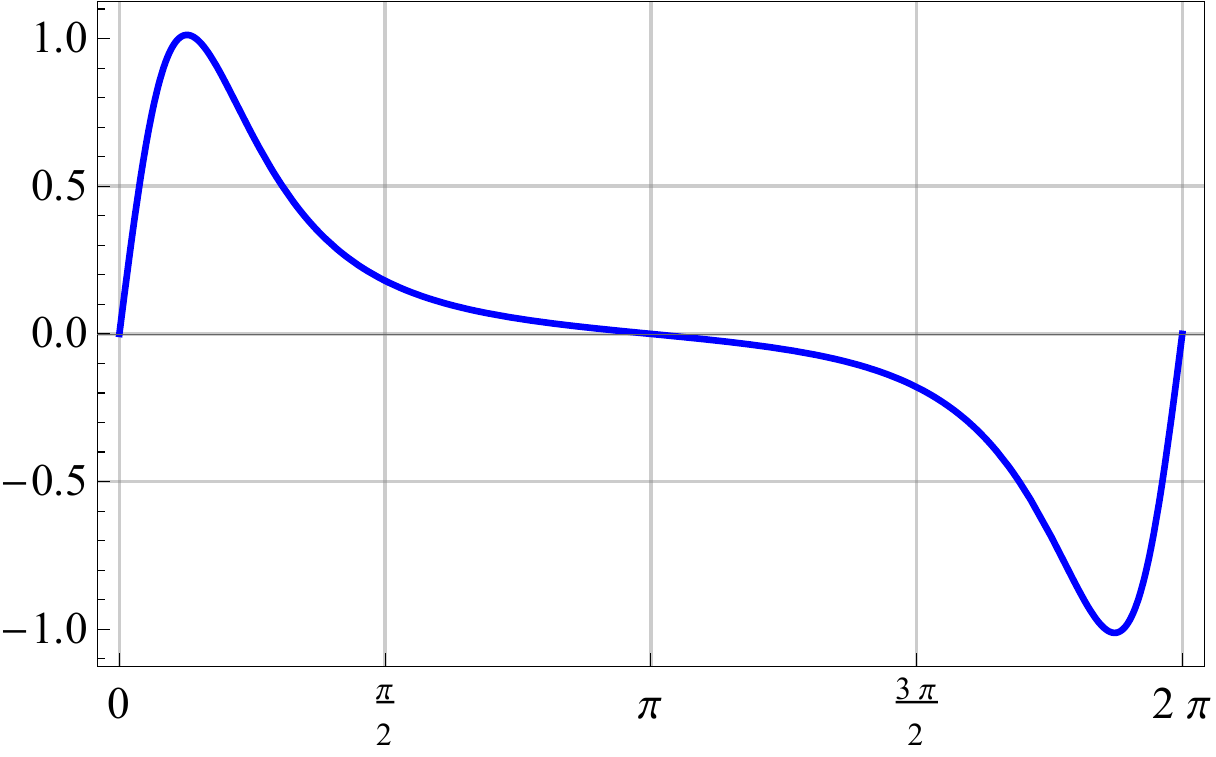}
\put(-84,113){\small $f'$}\\
\includegraphics[width=60mm]{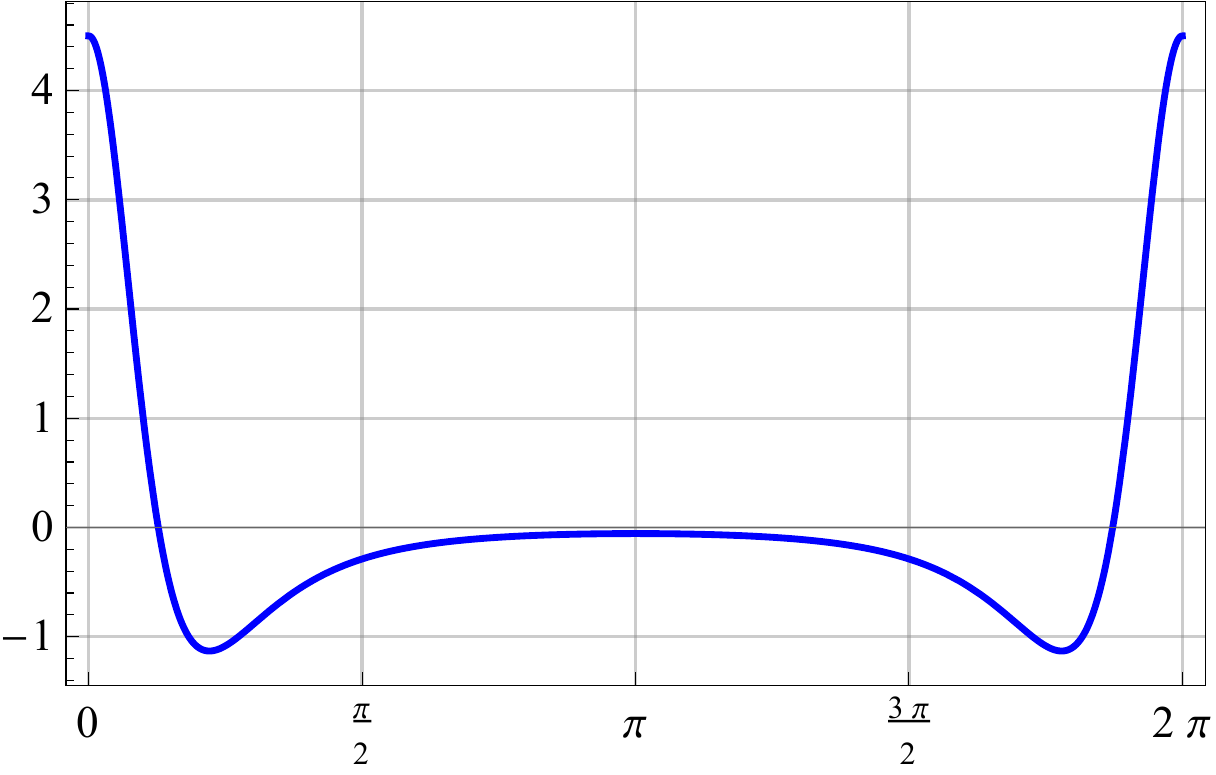}
\put(-84,113){\small $f''$}
\caption{The symbol $f$ and its first two derivatives for $\rho=\frac{1}{2}$.}\label{fg:f0}
\end{figure}

Our second symbol is given by
\[g(\si)\coloneqq 4\sin^{2}\Big(\frac{\si}{2}\Big).\]
The respective Fourier coefficients can be calculated as $\hat g_k=-1$ for $k=\pm1$, $\hat g_k=2$ for $k=0$, and $\hat g_k=0$ in any other case: we remind that the quoted symbol is related to the classical discrete Laplacian in one dimension. Hence $\|g\|_{\al}<\nf$ for any $\al>0$ and the remaining simple-loop conditions are easily verified in Figure~\ref{fg:f1}. Then $g\in\SL^{\al}$ for any $\al>0$. In this case we obtain $b_{g}(\si,s)=1$ and $\eta_{g}(s)=0$, and the eigenvalues of $T_{n}(g)$ can be obtained explicitly as $\la_{j}(T_{n}(g))=g(d_{j,n})$.

\begin{figure}[ht]
\centering
\includegraphics[width=60mm]{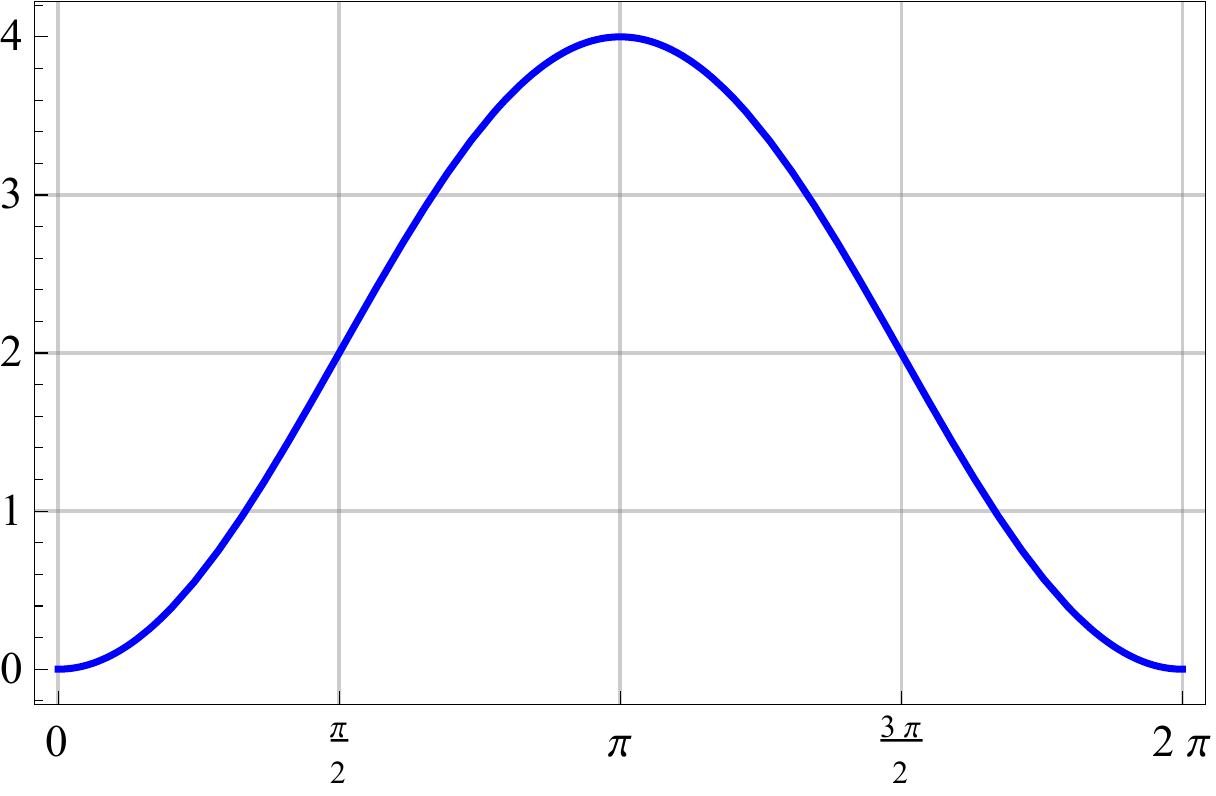}
\put(-84,115){\small $g$}
\hspace{2mm}
\includegraphics[width=60mm]{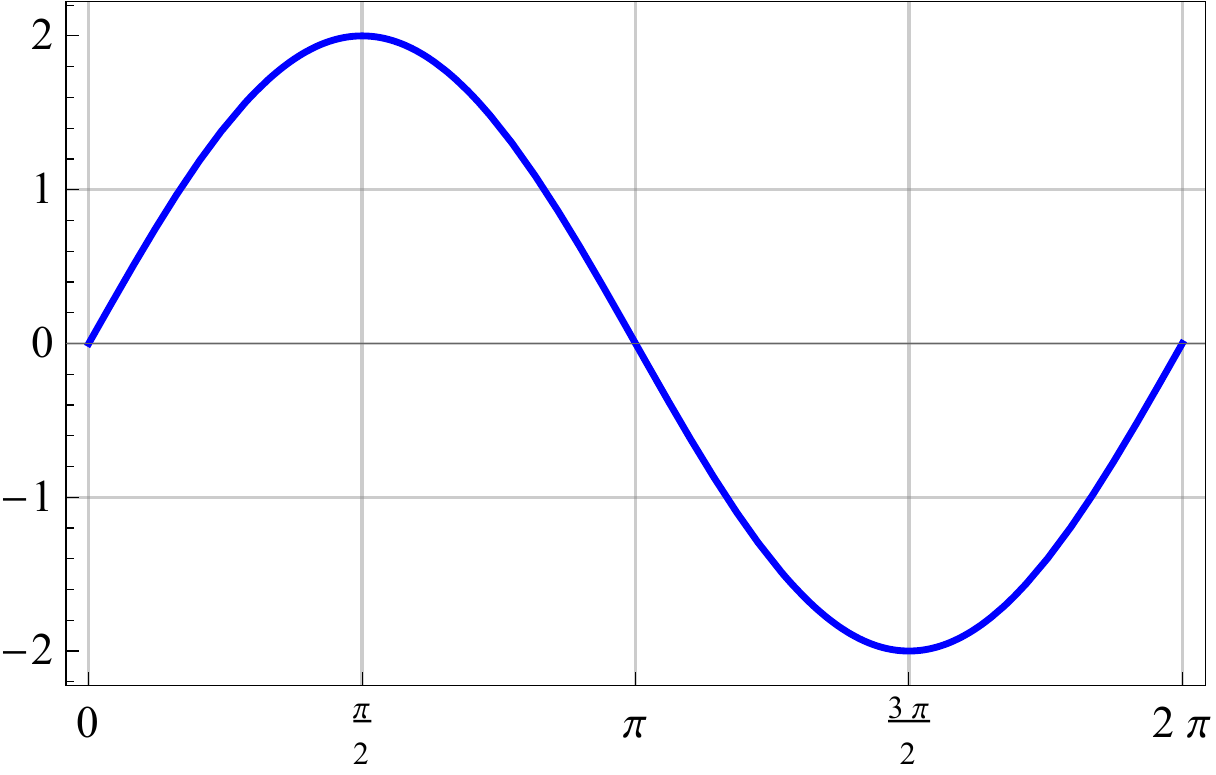}
\put(-84,113){\small $g'$}\\
\includegraphics[width=60mm]{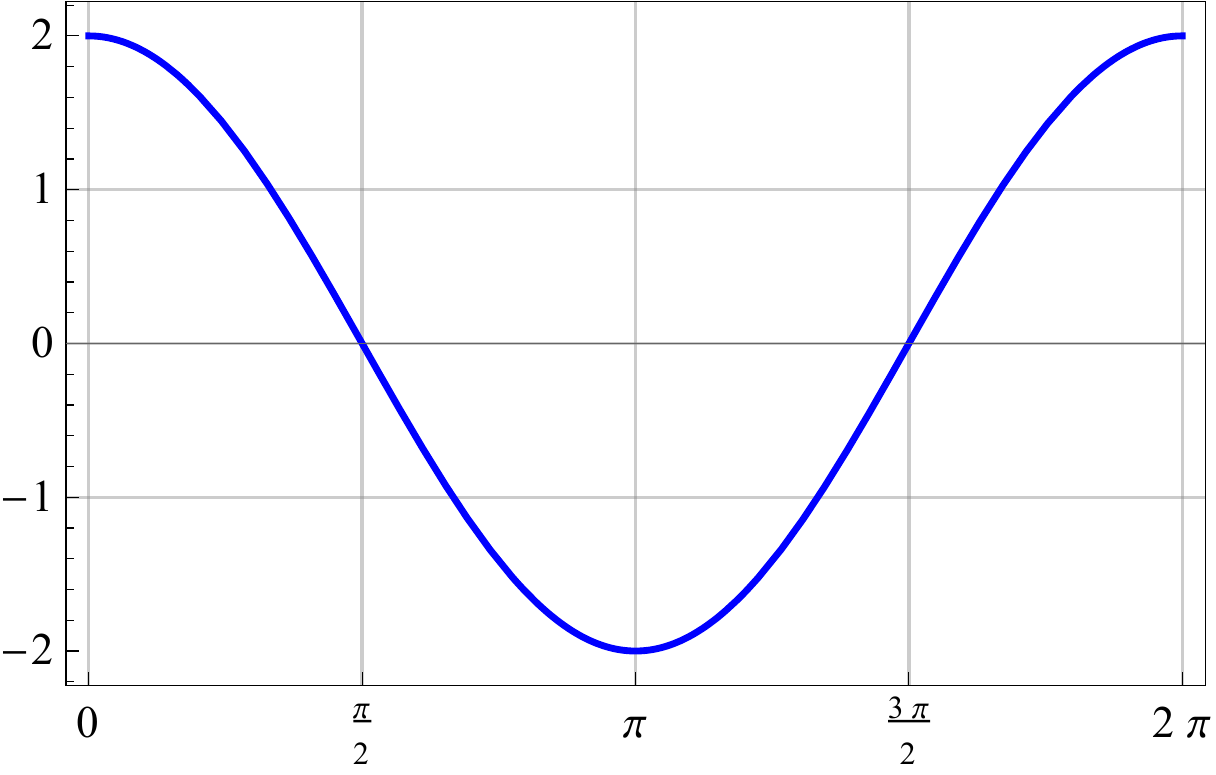}
\put(-84,113){\small $g''$}
\caption{The symbol $g$ and its first two derivatives.}\label{fg:f1}
\end{figure}

We start the numerical verification with the result in Lemma~\ref{lm:addition}. Since $W^{\al}$ is an algebra, the symbol $f+g$ clearly belongs to $\SL^{\al}$ for any $\al>0$. For $k=1,2,3$ let $\la_{j}^{(k)}(T_{n}(f))$ be the $k$th term approximation of $\la_{j}(T_{n}(f))$ given by our formulas. Specifically we find
\begin{eqnarray*}
\la_{j}^{(1)}(T_{n}(f+g))&=&\la_{j}(T_{n}(f))+\la_{j}(T_{n}(g)),\\
\la_{j}^{(2)}(T_{n}(f+g))&=&\la_{j}(T_{n}(f))+\la_{j}(T_{n}(g))+hQ_{1}(d_{j,n}),\\
\la_{j}^{(3)}(T_{n}(f+g))&=&\la_{j}(T_{n}(f))+\la_{j}(T_{n}(g))+hQ_{1}(d_{j,n})+h^{2}Q_{2}(d_{j,n}).
\end{eqnarray*}
Consider the error terms $\eps_{j,n}^{(k)}\coloneqq\la_{j}(T_{n}(f+g))-\la_{j}^{(k)}(T_{n}(f+g))$, and the corresponding maximal absolute error
\[\eps_{n}^{(k)}\coloneqq\max\{|\eps_{j,n}^{(k)}|\colon j=1,\ldots,n\}.\]
According to Lemma~\ref{lm:addition}, we must have $\eps_{j,n}^{(k)}=O(h^k)$ for $k=1,2,3$ uniformly in $j$, more specifically, the normalized errors $(n+1)^k\eps_{j,n}^{(k)}$ for $k=1,2$ must be ``close'' to the terms $Q_{1}$ and $Q_{2}$, respectively. Figure~\ref{fg:Q12} and Table~\ref{tb:Q12} show the data.

\begin{figure}[ht]
\centering
\includegraphics[width=75mm]{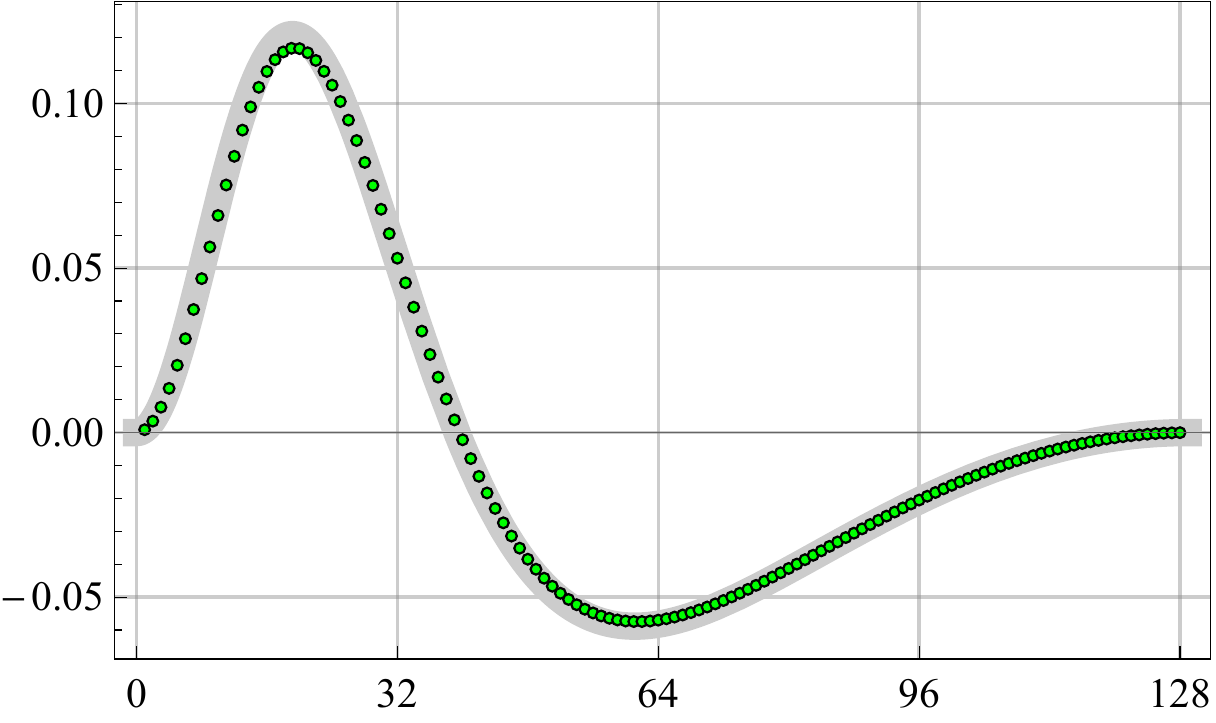}
\hspace{2mm}
\includegraphics[width=75mm]{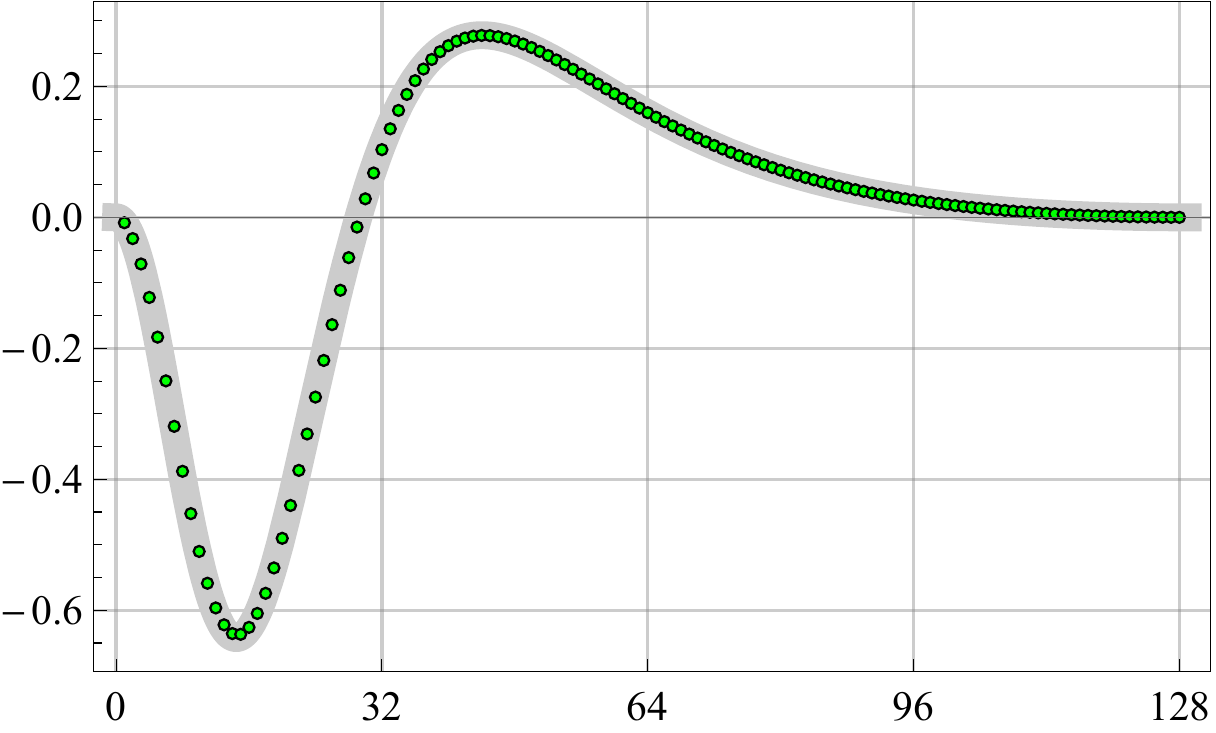}
\caption{Asymptotic expansion for the eigenvalues of $T_{n}(f+g)$. The light gray curve is the term $Q_k$ in Lemma~\ref{lm:addition}, and the green dots are the normalized errors $(n+1)^k\eps_{j,n}^{(k)}$ for $j=1,\ldots,n$ and $n=128$. The left graphic corresponds to the first term $k=1$ and the right one to the second $k=2$.}\label{fg:Q12}
\end{figure}

\begin{table}[ht]
\centering
{\footnotesize\begin{tabular}{|l|l|l|l|l|l|l|}
\hline
\multicolumn{1}{|c|}{$n$} & \multicolumn{1}{|c|}{$256$} & \multicolumn{1}{|c|}{$512$} & \multicolumn{1}{|c|}{$1024$} & \multicolumn{1}{|c|}{$2048$} & \multicolumn{1}{|c|}{$4096$} & \multicolumn{1}{|c|}{$8192$} \\ \hline\hline
$\eps_{n}^{(1)}$ & $4.6270\x10^{\text{-}4}$ & $2.3382\x10^{\text{-}4}$ & $1.1753\x10^{\text{-}4}$ & $5.8918\x10^{\text{-}5}$ & $2.9498\x10^{\text{-}5}$ & $1.4759\x10^{\text{-}5}$\\ \hline
$\hat\eps_{n}^{(1)}$ & $1.1891\x10^{\text{-}1}$& $1.1995\x10^{\text{-}1}$ & $1.2047\x10^{\text{-}1}$ & $1.2072\x10^{\text{-}1}$ & $1.2085\x10^{\text{-}1}$ & $1.2092\x10^{\text{-}1}$\\ \hline
$\eps_{n}^{(2)}$ & $9.6908\x10^{\text{-}6}$ & $2.4365\x10^{\text{-}6}$ & $6.1092\x10^{\text{-}7}$ & $1.5295\x10^{\text{-}7}$ & $3.8265\x10^{\text{-}8}$ & $9.5697\x10^{\text{-}9}$\\ \hline
$\hat\eps_{n}^{(2)}$ & $6.4007\x10^{\text{-}1}$ & $6.4120\x10^{\text{-}1}$ & $6.4184\x10^{\text{-}1}$ & $6.4214\x10^{\text{-}1}$ & $6.4229\x10^{\text{-}1}$ & $6.4237\x10^{\text{-}1}$\\ \hline
$\eps_{n}^{(3)}$ & $6.8729\x10^{\text{-}8}$ & $8.6728\x10^{\text{-}9}$ & $1.0883\x10^{\text{-}9}$ & $1.3634\x10^{\text{-}10}$ & $1.7037\x10^{\text{-}11}$ & $ 2.1231\x10^{\text{-}12}$\\ \hline
$\hat\eps_{n}^{(3)}$ & $1.1666\x10^{0}$ & $1.1709\x10^{0}$ & $1.1720\x10^{0}$ & $1.1722\x10^{0}$ & $1.1716\x10^{0}$ & $1.1676\x10^{0}$\\ \hline
\end{tabular}}
\vspace{2mm}
\caption{The maximum errors $\eps_{n}^{(k)}$ and maximum normalized errors $\hat\eps_{n}^{(k)}\coloneqq (n+1)^k\eps_{n}^{(k)}$ for $k=1,2,3$ and different values of $n$, corresponding to the asymptotic expansion for the eigenvalues of $T_{n}(f+g)$ in Lemma \ref{lm:addition}.}\label{tb:Q12}
\end{table}

For the numerical verification of Theorem~\ref{th:SLh}, we need to calculate the singular integral in $\psi$, which for this example can be simplified to
\[\psi(s)=\frac{\sin(s)}{\pi}\dint_0^{2\pi}\frac{1}{f(\si)-f(s)}\d \si.\]
For the $k$th term approximation we get this time
\begin{eqnarray*}
\la_{j}^{(1)}&=&f(d_{j,n}),\\
\la_{j}^{(2)}&=&f(d_{j,n})+\Psi_{1}(d_{j,n})h,\\
\la_{j}^{(3)}&=&f(d_{j,n})+\Psi_{1}(d_{j,n})h+\Psi_{2}(d_{j,n})h^{2}.
\end{eqnarray*}
According to the results in Theorem~\ref{th:SLh}, we must have $\eps_{j,n}^{(k)}=O(h^k)$ uniformly in $j=1,\ldots,n$ for $k=1,2,3$, more specifically the normalized errors $(n+1)^k\eps_{j,n}^{(k)}$ for $k=1,2$ are expected to be ``close'' to $\Psi_{1}$ and $\Psi_{2}$, respectively. The Figure~\ref{fg:Psi12} and Table~\ref{tb:Psi12} show the data.

\begin{figure}[ht]
\centering
\includegraphics[width=71mm]{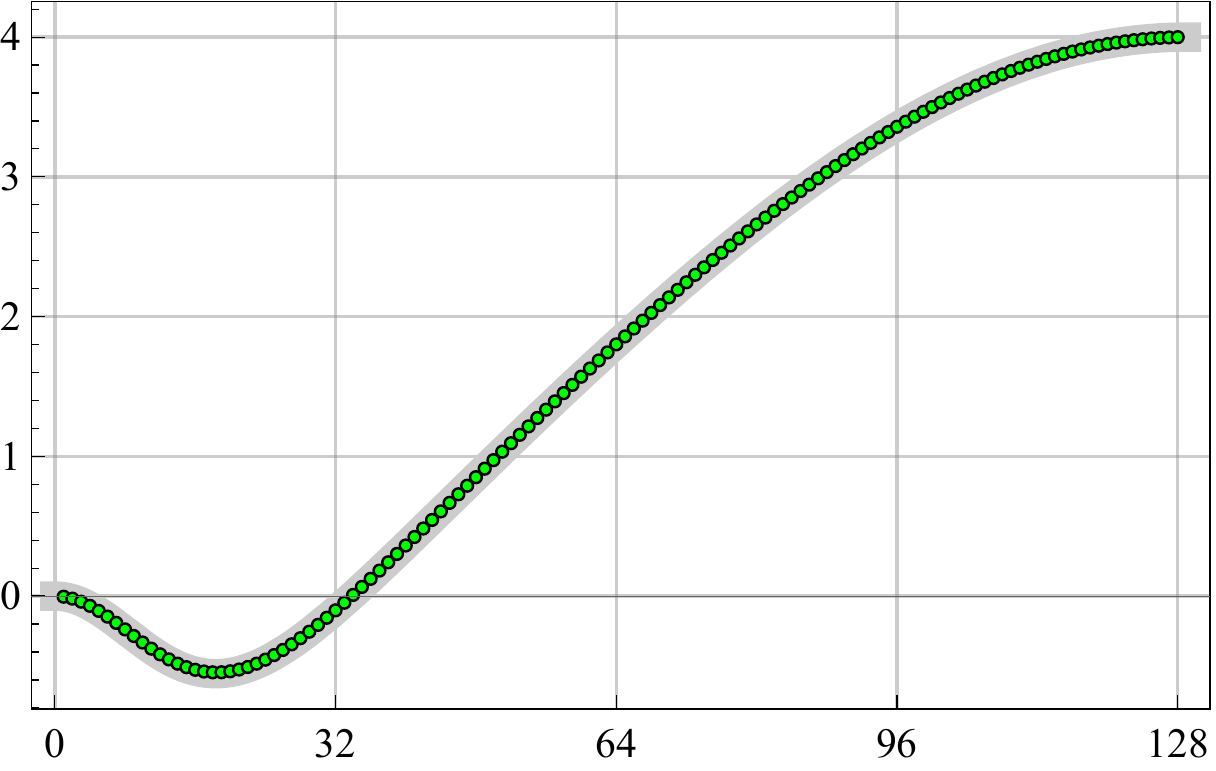}
\hspace{2mm}
\includegraphics[width=75mm]{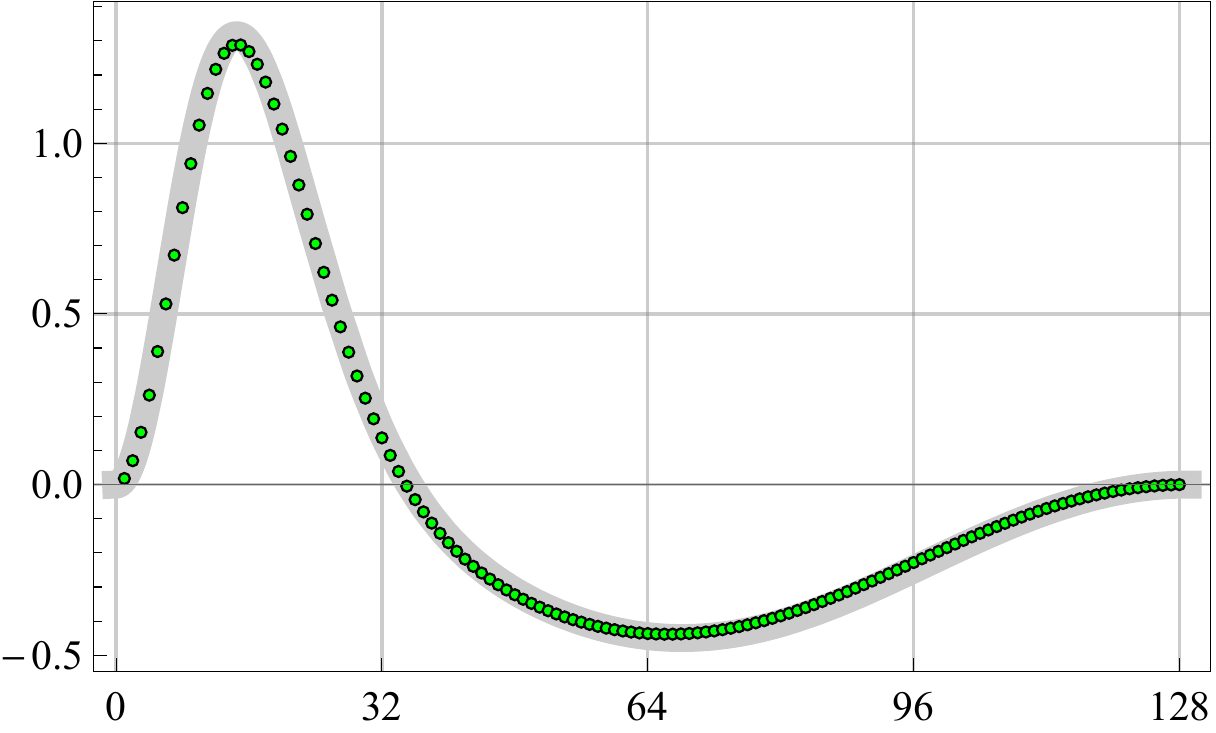}
\caption{Asymptotic expansion for the eigenvalues of $T_{n}(f+gh)$. The light gray curve is the term $\Psi_k$ in Theorem~\ref{th:SLh}, and the green dots are the normalized errors $(n+1)^k\eps_{j,n}^{(k)}$ for $j=1,\ldots,n$ and $n=128$. The left graphic corresponds to the first term $k=1$ and the right one to the second $k=2$.}\label{fg:Psi12}
\end{figure}

\begin{table}[ht]
\centering
{\footnotesize\begin{tabular}{|l|l|l|l|l|l|l|}
\hline
\multicolumn{1}{|c|}{$n$} & \multicolumn{1}{|c|}{$256$} & \multicolumn{1}{|c|}{$512$} & \multicolumn{1}{|c|}{$1024$} & \multicolumn{1}{|c|}{$2048$} & \multicolumn{1}{|c|}{$4096$} & \multicolumn{1}{|c|}{$8192$} \\ \hline\hline
$\eps_{n}^{(1)}$ & $1.5564\x10^{\text{-}2}$ & $7.7972\x10^{\text{-}3}$ & $3.9024\x10^{\text{-}3}$ & $1.9522\x10^{\text{-}3}$ & $9.7632\x10^{\text{-}4}$ & $4.8822\x10^{\text{-}4}$\\ \hline
$\hat\eps_{n}^{(1)}$ & $3.9998\x10^{0}$& $4.0000\x10^{0}$ & $4.0000\x10^{0}$ & $4.0000\x10^{0}$ & $4.0000\x10^{0}$ & $4.0000\x10^{0}$\\ \hline
$\eps_{n}^{(2)}$ & $1.9728\x10^{\text{-}5}$ & $4.9765\x10^{\text{-}6}$ & $1.2498\x10^{\text{-}6}$ & $3.1315\x10^{\text{-}7}$ & $7.8377\x10^{\text{-}8}$ & $1.9605\x10^{\text{-}8}$\\ \hline
$\hat\eps_{n}^{(2)}$ & $1.3030\x10^{0}$ & $1.3096\x10^{0}$ & $1.3130\x10^{0}$ & $1.3147\x10^{0}$ & $1.3156\x10^{0}$ & $1.3160\x10^{0}$\\ \hline
$\eps_{n}^{(3)}$ & $2.0796\x10^{\text{-}7}$ & $2.6386\x10^{\text{-}8}$ & $3.3384\x10^{\text{-}9}$ & $4.2142\x10^{\text{-}10}$ & $5.3166\x10^{\text{-}11}$ & $ 6.7005\x10^{\text{-}12}$\\ \hline
$\hat\eps_{n}^{(3)}$ & $3.5301\x10^{0}$ & $3.5623\x10^{0}$ & $3.5950\x10^{0}$ & $3.6253\x10^{0}$ & $3.6562\x10^{0}$ & $3.6850\x10^{0}$\\ \hline
\end{tabular}}
\vspace{2mm}
\caption{The maximum errors $\eps_{n}^{(k)}$ and maximum normalized errors $\hat\eps_{n}^{(k)}\coloneqq (n+1)^k\eps_{n}^{(k)}$ for $k=1,2,3$ and different values of $n$, corresponding to the asymptotic expansion for the eigenvalues of $T_{n}(f+gh)$ in Theorem \ref{th:SLh}.}\label{tb:Psi12}
\end{table}

For the numerical verification of Theorem~\ref{th:SLhh}, we need to calculate the singular integral in $\ph$, which for this example can be simplified to
\[\ph(s)=\frac{\sin(s)}{\pi}\dint_0^{2\pi}\frac{1}{f(\si)-f(s)+2\cos(s)-2\cos(\si)}\d \si.\]
Taking into account that the logarithm is relatively small for the matrix sizes considered, we arrange the $k$th term approximation in a different way this time
\begin{eqnarray*}
\la_{j}^{(1)}&=&f(d_{j,n})+g(d_{j,n}),\\
\la_{j}^{(2)}&=&f(d_{j,n})+g(d_{j,n})+\Ga_{1,1}(d_{j,n})h\log(h)+\Ga_{1,0}(d_{j,n})h,\\
\la_{j}^{(3)}&=&f(d_{j,n})+g(d_{j,n})+\Ga_{1,1}(d_{j,n})h\log(h)+\Ga_{1,0}(d_{j,n})h\\
&&+\Ga_{2,2}(d_{j,n})h^{2}\log^{2}(h)+\Ga_{2,1}(d_{j,n})h^{2}\log(h)+\Ga_{2,0}(d_{j,n})h^{2}.
\end{eqnarray*}
According to the results in Theorem~\ref{th:SLh}, we have the bound $\eps_{j,n}^{(k)}=O(h^k|\log^k(h)|)$ uniformly in $j=1,\ldots,n$ for $k=1,2,3$, more specifically the normalized errors $\frac{(n+1)^k}{\log^k(n+1)}\eps_{j,n}^{(k)}$ for $k=1,2$ are expected to be ``close'' to the functions $\Ga_{1,1}+\Ga_{1,0}\frac{1}{\log(h)}$ and $\Ga_{2,2}+\Ga_{2,1}\frac{1}{\log(h)}+\Ga_{2,0}\frac{1}{\log^{2}(h)}$, respectively. The Figure~\ref{fg:Ga12} and Table~\ref{tb:Ga12} show the data.

\begin{figure}[ht]
\centering
\includegraphics[width=73mm]{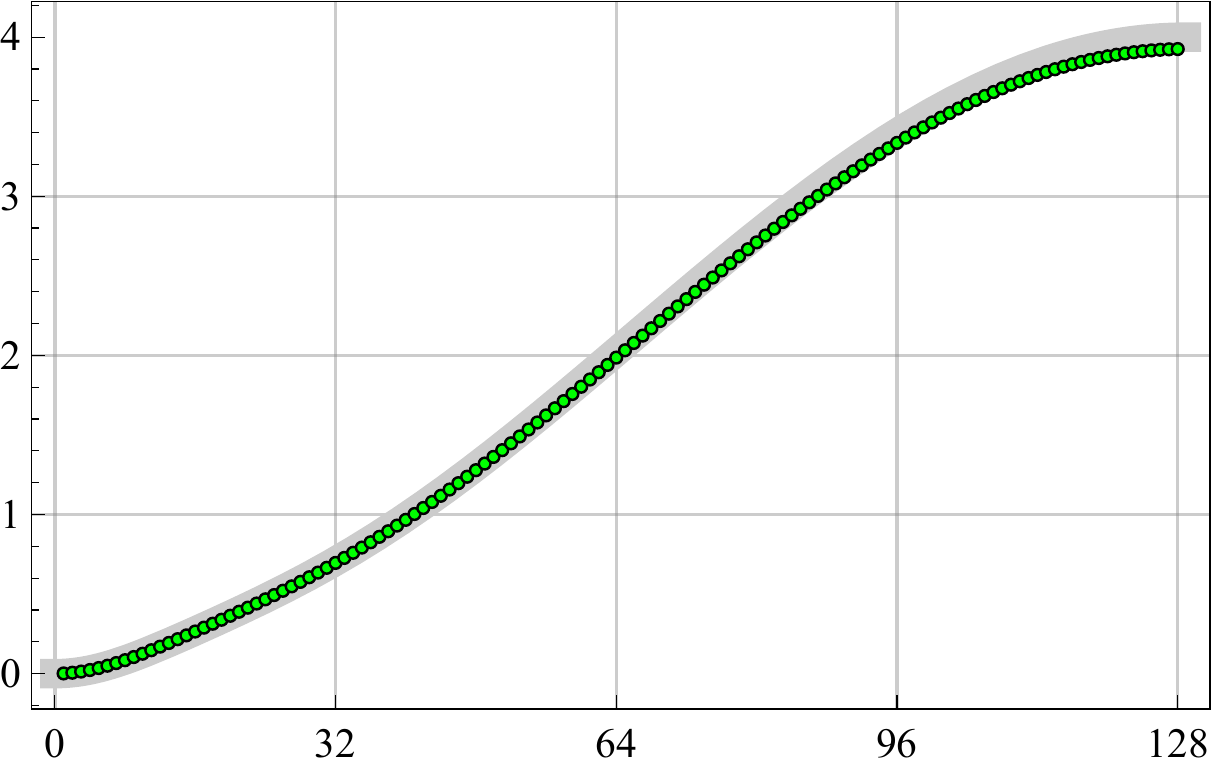}
\hspace{2mm}
\includegraphics[width=75mm]{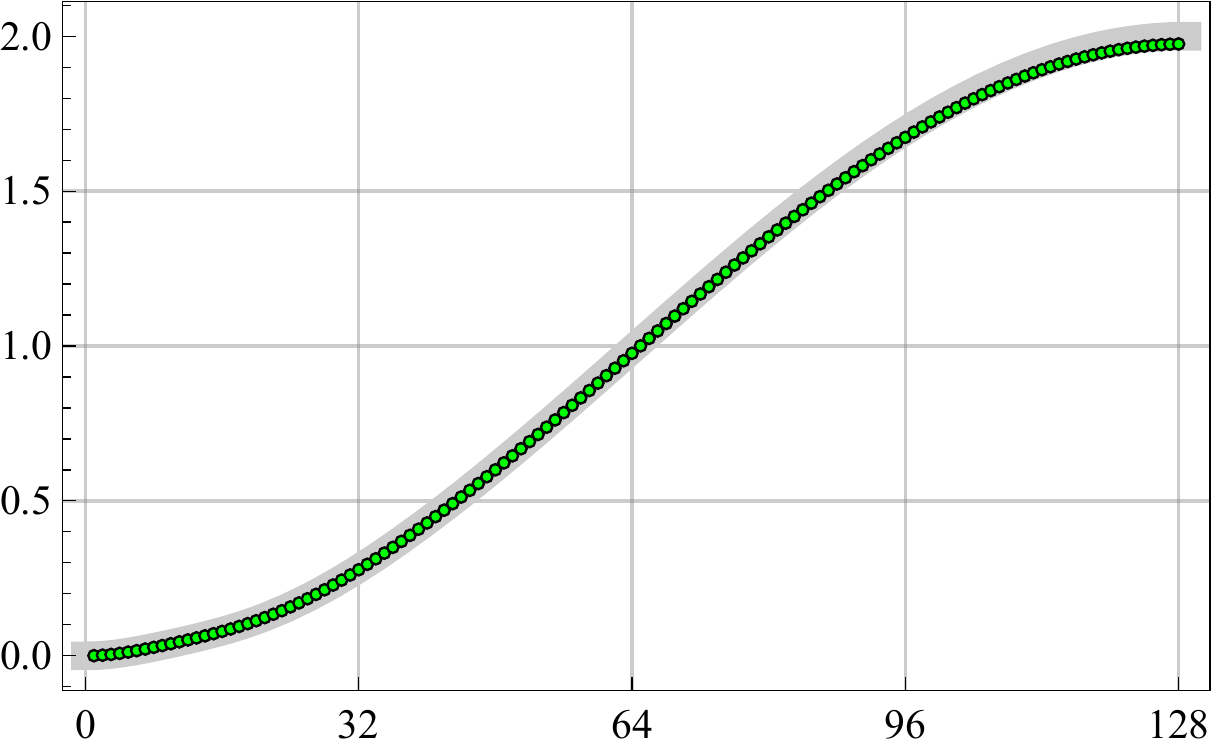}
\caption{Asymptotic expansion for the eigenvalues of the matrix $T_{n}(f+gh^h)$. The light gray curve is the term $\Ga_{1,1}+\Ga_{1,0}\frac{1}{\log(h)}$ (left) and $\Ga_{2,2}+\Ga_{2,1}\frac{1}{\log(h)}+\Ga_{2,0}\frac{1}{\log^{2}(h)}$ (right) in Theorem~\ref{th:SLhh}, and the green dots are the normalized errors $\frac{(n+1)^k}{\log^k(n+1)}\eps_{j,n}^{(k)}$ for $j=1,\ldots,n$ and $n=128$. The left graphic corresponds to the first term $k=1$ and the right one to the second $k=2$.}\label{fg:Ga12}
\end{figure}

\begin{table}[ht]
\centering
{\footnotesize\begin{tabular}{|l|l|l|l|l|l|l|}
\hline
\multicolumn{1}{|c|}{$n$} & \multicolumn{1}{|c|}{$256$} & \multicolumn{1}{|c|}{$512$} & \multicolumn{1}{|c|}{$1024$} & \multicolumn{1}{|c|}{$2048$} & \multicolumn{1}{|c|}{$4096$} & \multicolumn{1}{|c|}{$8192$} \\ \hline\hline
$\eps_{n}^{(1)}$ & $8.5438\x10^{\text{-}2}$ & $4.8362\x10^{\text{-}2}$ & $2.6962\x10^{\text{-}2}$ & $1.4858\x10^{\text{-}2}$ & $8.1128\x10^{\text{-}3}$ & $4.3970\x10^{\text{-}3}$\\ \hline
$\hat\eps_{n}^{(1)}$ & $3.9570\x10^{0}$ & $3.9757\x10^{0}$ & $3.9865\x10^{0}$ & $3.9926\x10^{0}$ & $3.9959\x10^{0}$ & $3.9978\x10^{0}$\\ \hline
$\eps_{n}^{(2)}$ & $9.2570\x10^{\text{-}4}$ & $2.9474\x10^{\text{-}4}$ & $9.1280\x10^{\text{-}5}$ & $2.7663\x10^{\text{-}5}$ & $8.2384\x10^{\text{-}6}$ & $2.4184\x10^{\text{-}6}$\\ \hline
$\hat\eps_{n}^{(2)}$ & $1.9856\x10^{0}$ & $1.9919\x10^{0}$ & $1.9955\x10^{0}$ & $1.9975\x10^{0}$ & $1.9986\x10^{0}$ & $1.9993\x10^{0}$\\ \hline
$\eps_{n}^{(3)}$ & $1.0171\x10^{\text{-}5}$ & $1.8568\x10^{\text{-}6}$ & $3.2748\x10^{\text{-}7}$ & $5.5520\x10^{\text{-}8}$ & $9.3579\x10^{\text{-}9}$ & $1.5453\x10^{\text{-}9}$\\ \hline
$\hat\eps_{n}^{(3)}$ & $1.0104\x10^{0}$ & $1.0316\x10^{0}$ & $1.0585\x10^{0}$ & $1.0773\x10^{0}$ & $1.1182\x10^{0}$ & $1.1615\x10^{0}$\\ \hline
\end{tabular}}
\vspace{2mm}
\caption{The maximum errors $\eps_{n}^{(k)}$ and maximum normalized errors $\hat\eps_{n}^{(k)}\coloneqq(n+1)^k\eps_{n}^{(k)}$ for the $k=1,2,3$ and different values of $n$, corresponding to the asymptotic expansion for the eigenvalues of $T_{n}(f+gh^h)$ in Theorem \ref{th:SLhh}.}\label{tb:Ga12}
\end{table}
\clearpage

\subsection{A substantial improvement of the algorithm in \cite{EkGa19}}

For $k\in\bZ_{+}$ let $f_{k}(\tht)\coloneqq(2-2\cos(\tht))^{k}$ and $\al_{k}\in\bR$. Consider the symbol
\begin{equation}\label{eq:Fn}
F_{n}(\tht)\coloneqq f_{2}(\tht)+\al_{1}f_{1}(\tht)h^{2}+\al_{0}f_{0}(\tht)h^{4}.
\end{equation}
The related Toeplitz matrices $T_{n}(F_{n})$ appear when discretizing differential equations with the Finite Differences method.
\begin{figure}[ht]
\centering
\includegraphics[width=60mm]{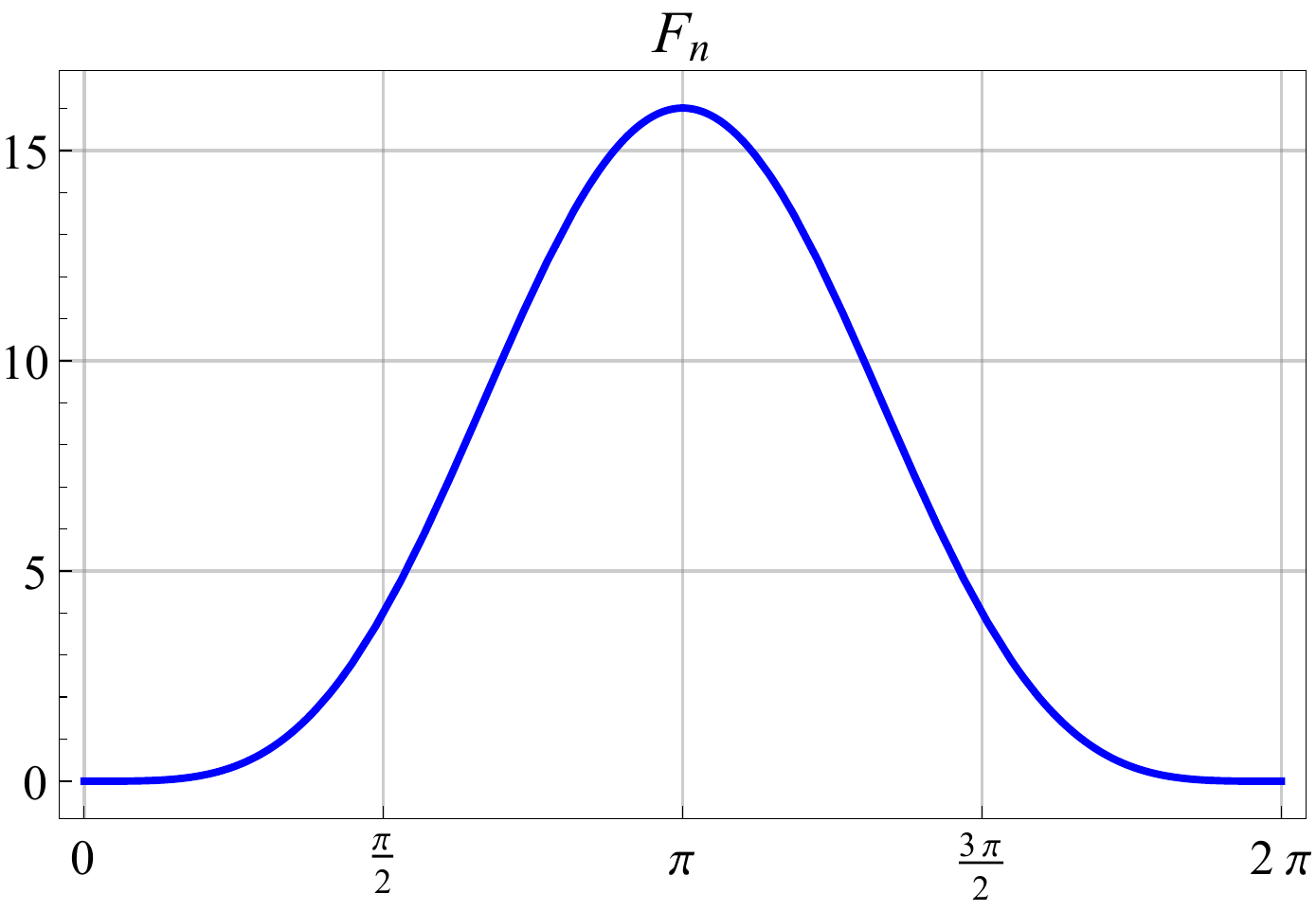}
\includegraphics[width=60mm]{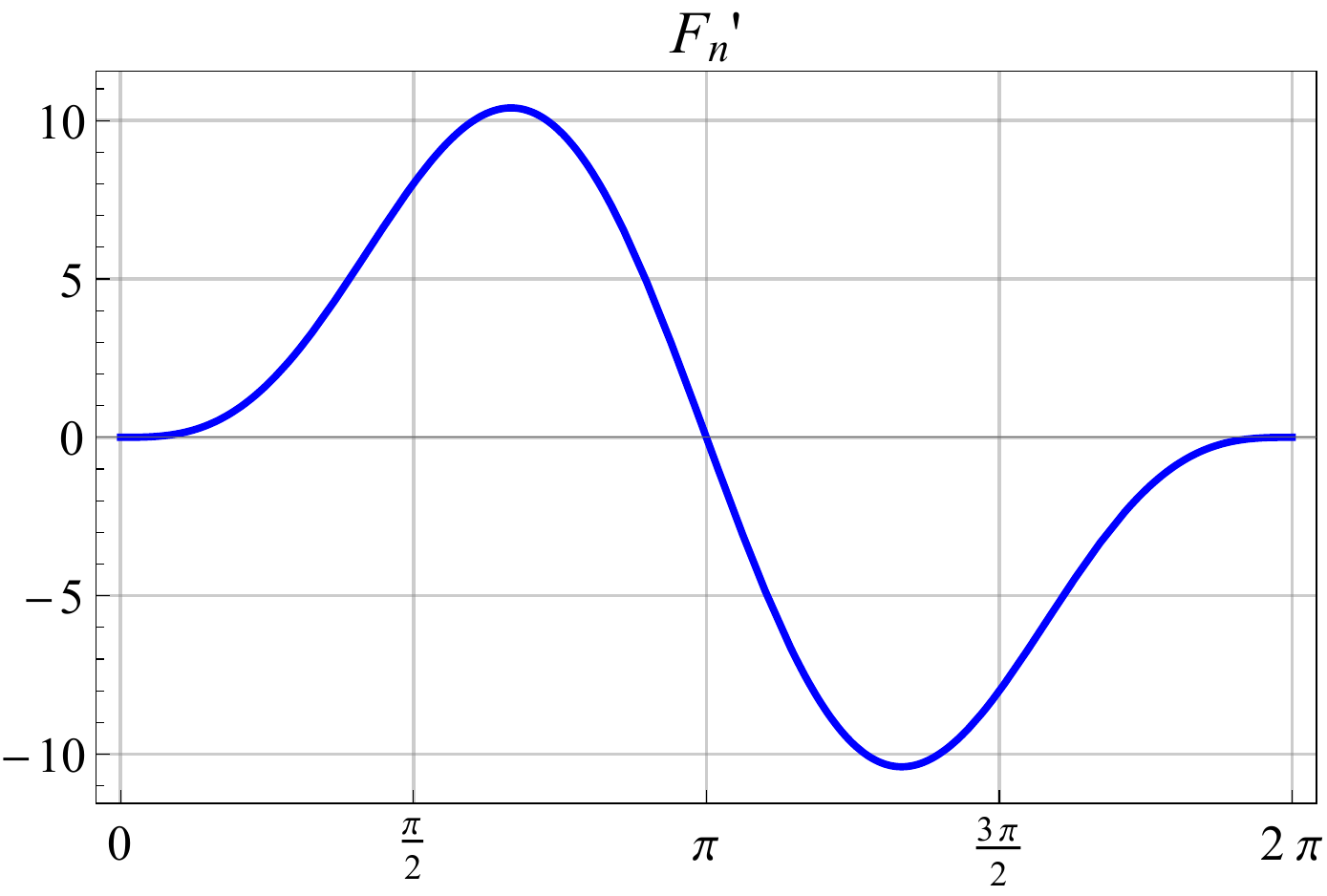}
\includegraphics[width=60mm]{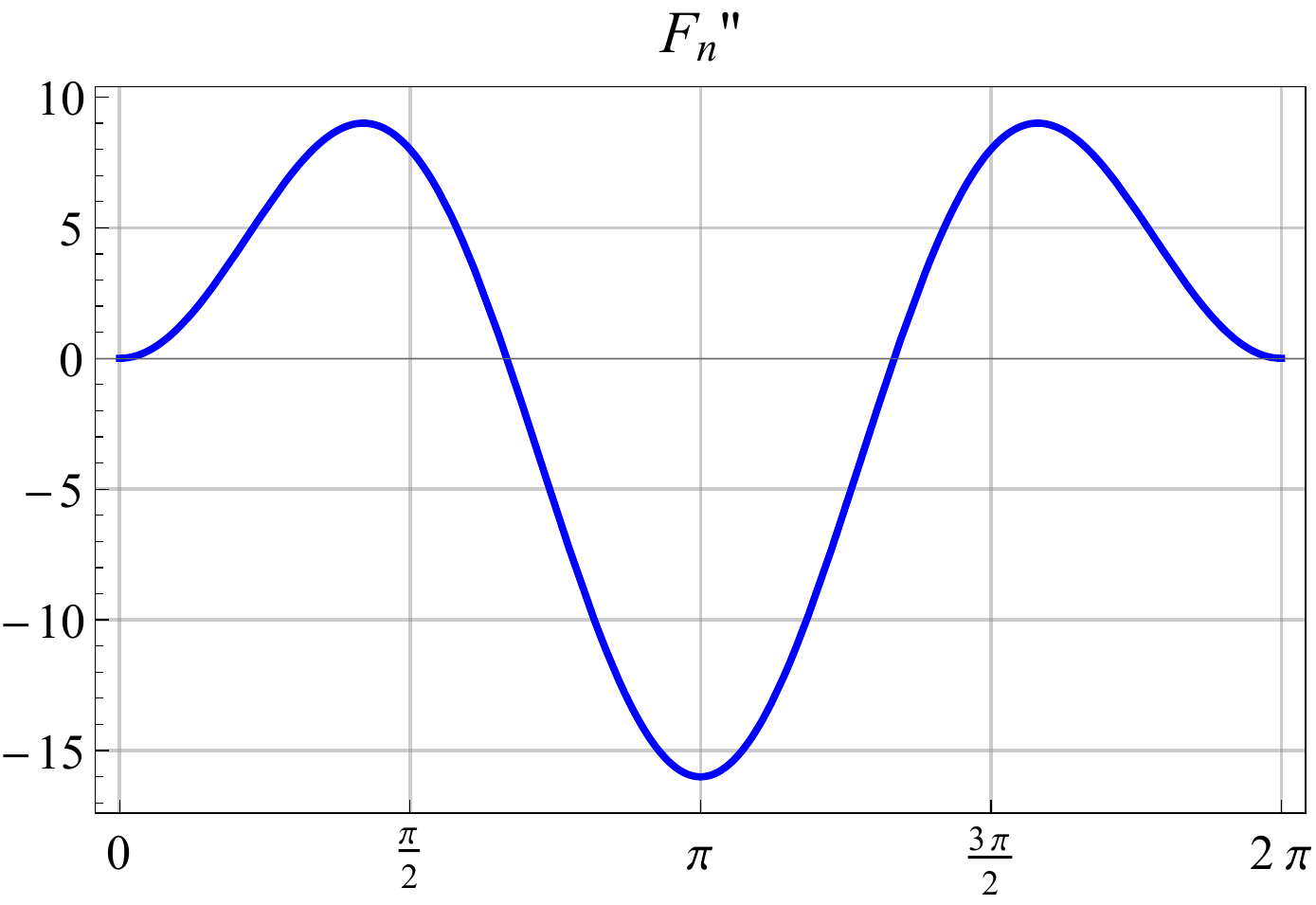}
\caption{The symbol $F_{n}$ and its first two derivatives.}
\end{figure}
The functions $f_{k}$ with $k\ne2$, do not belong to $\SL^{\al}$ for any $\al$, thus $F_{n}$ do not fully satisfy our hypothesis and we can not apply our theoretical results, but in \cite{EkFu18b,EkGa19,EkGa18} there is numerical evidence suggesting that, nevertheless we can expect an eigenvalue expansion of the form
\begin{equation}\label{eq:AEx}
\la_{j}(T_{n}(F_{n}))\sim f_{2}(d_{j,n})+\sum_{\ell=1}^{\nf} c_{\ell}(d_{j,n})h^{\ell},
\end{equation}
where $\sim$ means asymptotic expansion. We now proceed to determine the continuous functions $c_{\ell}$ following the algorithm proposed in \cite{EkFu18b,EkGa19,EkGa18}. This algorithm is a clever interplay between extrapolation and interpolation. For given natural numbers $k$ and $n_{0}$, the extrapolation step, gives us a numerical approximation of $c_{\ell}$ $(\ell=1,\ldots,k-1)$ at the regular mesh $\pi j h_{0}$ $(j=1,\ldots,n_{0})$ of the interval $[0,\pi]$. Then, in the interpolation step, we use certain polynomial interpolation to obtain the value of $c_{\ell}$ at any point in the interval $[0,\pi]$.

We noticed that this algorithm can produce large errors when evaluating $c_{\ell}$ close to the extreme points $\{0,\pi\}$, i.e. $10^{3}$ bigger that the rest. As we can see in Theorems \ref{th:SLh} and \ref{th:SLhh}, the functions $c_{\ell}$ are a sum of products of $\eta$, the symbol, and its derivatives, and we must have $\eta(0)=\eta(\pi)=0$ (see \eqref{eq:eta}). Therefore, considering the nature of the symbol \eqref{eq:Fn} we have
\begin{center}
\begin{tabular}{ll}
$c_{2}(0)=\al_{1}f_{1}(0)=0$, & $c_{2}(\pi)=\al_{1}f_{1}(\pi)=4\al_{1}$,\\
$c_{4}(0)=\al_{0}f_{0}(0)=0$, & $c_{4}(\pi)=\al_{0}f_{0}(\pi)=\al_{0}$,
\end{tabular}
\end{center}
while $c_{\ell}(0)=c_{\ell}(\pi)=0$ for $\ell\ne2,4$. Hence we propose here to include those values in the interpolation step. Our numerical results reveal that the errors in the extreme points almost disappear and that the approximation of the functions $c_{\ell}$ improves in such a way that the   respective errors behave better when $\ell$ increases (see Tables \ref{tb:NF+} and \ref{tb:NF-}).
We emphasize that this contribution represents a substantial improvement to the algorithm presented in \cite{EkGa19} and it is a general idea to be used in various contexts. 

More specifically, let $\tilde \la_{j}^{(k)}(T_{n}(F_{n}))$ be the $k$th-term numerical approximation of the eigenvalue $\la_{j}(T_{n}(F_{n}))$, given by \eqref{eq:AEx}, that is
\[\tilde \la_{j}^{(k)}(T_{n}(F_{n}))\coloneqq f_{2}(d_{j,n})+\sum_{\ell=1}^{k-1} \tilde c_{\ell}(d_{j,n})h^{\ell},\]
where the functions $\tilde c_{\ell}$ are obtained with the numerical algorithm. Let
\[\eps_{j,n}^{(k)}\coloneqq|\la_{j}(T_{n}(F_{n}))-\tilde \la_{j}^{(k)}(T_{n}(F_{n}))|,\qquad
\eps_{n}^{(k)}\coloneqq\max\{\eps_{j,n}^{(k)}\colon j=1,\ldots,n\},\]
be the corresponding errors. According to \cite[Th.3]{EkGa19} we will have $\eps_{n}^{(k)}=O(h_{0}^{k}h)$, thus let $\hat\eps_{n}^{(k)}\coloneqq(n_{0}+1)^{k}(n+1)\eps_{n}^{(k)}$ be the respective normalized error. The Figures \ref{fg:NF+}, \ref{fg:NF-}, and the Tables \ref{tb:NF+}, \ref{tb:NF-}, show the data. 

\begin{table}[ht]
\centering
{\footnotesize\begin{tabular}{|l|l|l|l|l|l|l|}
\hline
\multicolumn{1}{|c|}{$n$} & \multicolumn{1}{|c|}{$256$} & \multicolumn{1}{|c|}{$512$} & \multicolumn{1}{|c|}{$1024$} & \multicolumn{1}{|c|}{$2048$} & \multicolumn{1}{|c|}{$4096$} & \multicolumn{1}{|c|}{$8192$} \\ \hline\hline
$\eps_{n}^{(1)}$ & $1.6359\x10^{\text{-}2}$ & $8.1806\x10^{\text{-}3}$ & $4.0904\x10^{\text{-}3}$ & $2.0453\x10^{\text{-}3}$ & $1.0227\x10^{\text{-}3}$ & $5.1133\x10^{\text{-}4}$\\ \hline
$\hat\eps_{n}^{(1)}$ & $4.2464\x10^{2}$ & $4.2386\x10^{2}$ & $4.2327\x10^{2}$ & $4.2317\x10^{2}$ & $4.2312\x10^{2}$ & $4.2656\x10^{2}$\\ \hline
$\eps_{n}^{(2)}$ & $8.0281\x10^{\text{-}5}$ & $2.0194\x10^{\text{-}5}$ & $5.0639\x10^{\text{-}6}$ & $1.2679\x10^{\text{-}6}$ & $3.1722\x10^{\text{-}7}$ & $7.9337\x10^{\text{-}8}$\\ \hline
$\hat\eps_{n}^{(2)}$ & $2.1047\x10^{2}$ & $1.0568\x10^{2}$ & $5.2948\x10^{1}$ & $2.6502\x10^{1}$ & $1.3258\x10^{1}$ & $6.6307\x10^{0}$\\ \hline
$\eps_{n}^{(3)}$ & $1.0280\x10^{\text{-}7}$ & $1.2951\x10^{\text{-}8}$ & $1.6253\x10^{\text{-}9}$ & $2.0356\x10^{\text{-}10}$ & $2.5467\x10^{\text{-}11}$ & $3.1833\x10^{\text{-}12}$\\ \hline
$\hat\eps_{n}^{(3)}$ & $2.7221\x10^{1}$ & $6.8452\x10^{0}$ & $1.7164\x10^{0}$ & $4.2974\x10^{\text{-}1}$ & $1.0750\x10^{\text{-}1}$ & $2.6871\x10^{\text{-}2}$\\ \hline
$\eps_{n}^{(4)}$ & $3.2772\x10^{\text{-}9}$ & $1.9387\x10^{\text{-}10}$ & $1.5106\x10^{\text{-}11}$ & $6.0169\x10^{\text{-}12}$ & $3.2499\x10^{\text{-}12}$ & $1.5816\x10^{\text{-}12}$\\ \hline
$\hat\eps_{n}^{(4)}$ & $8.7646\x10^{1}$ & $1.0350\x10^{1}$ & $1.6114\x10^{0}$ & $1.2829\x10^{0}$ & $1.3855\x10^{0}$ & $1.3484\x10^{0}$\\ \hline
\end{tabular}}
\vspace{2mm}
\caption{The maximum eigenvalue errors $\eps_{n}^{(k)}$ and normalized errors $\hat\eps_{n}^{(k)}$ for an expansion of $k=1,\ldots,4$ terms, corresponding to the symbol $(2-2\cos(\tht))^{2}+3(2-2\cos(\tht))h^2+2h^{4}$ (see \eqref{eq:Fn}) with different values of $n$ and a regular mesh of $n_{0}=100$ points.}\label{tb:NF+}
\end{table}

\begin{figure}[ht]
\centering
\includegraphics[width=135mm]{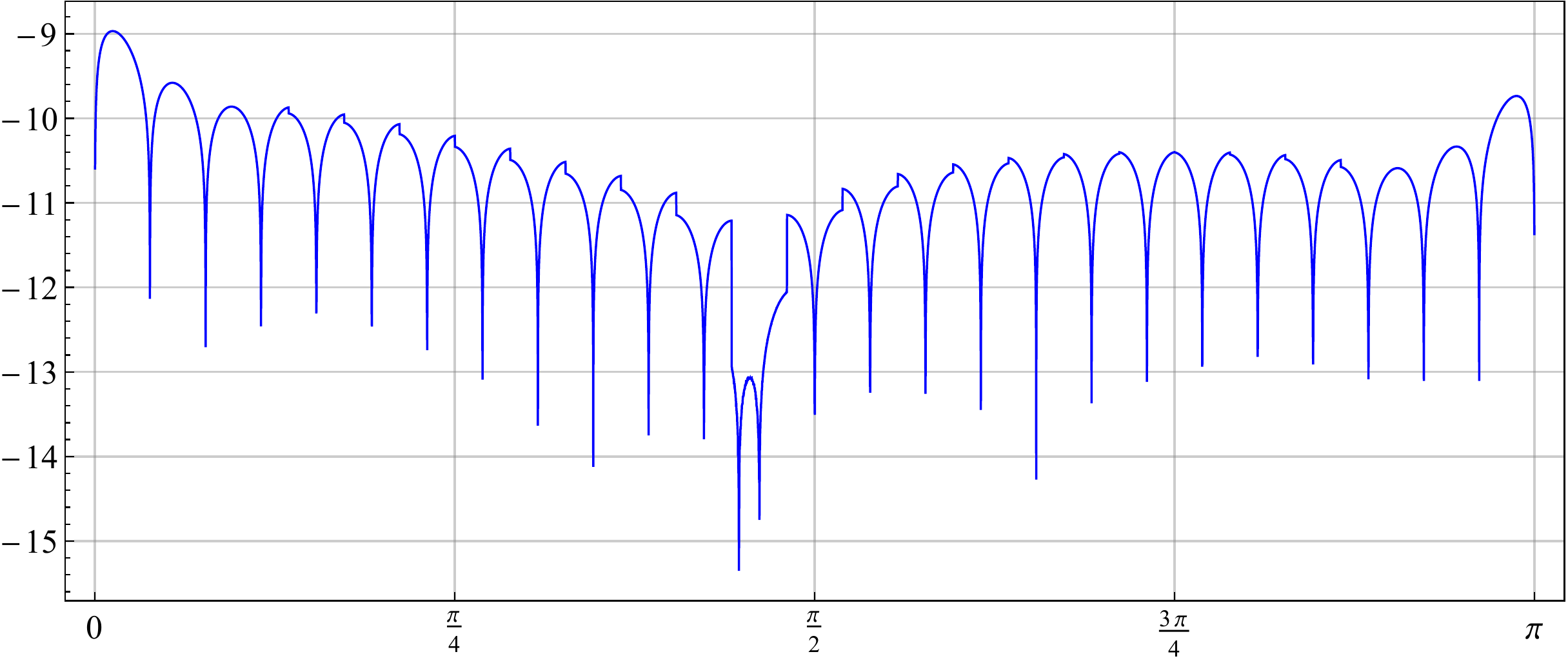}
\bigskip\\
\includegraphics[width=135mm]{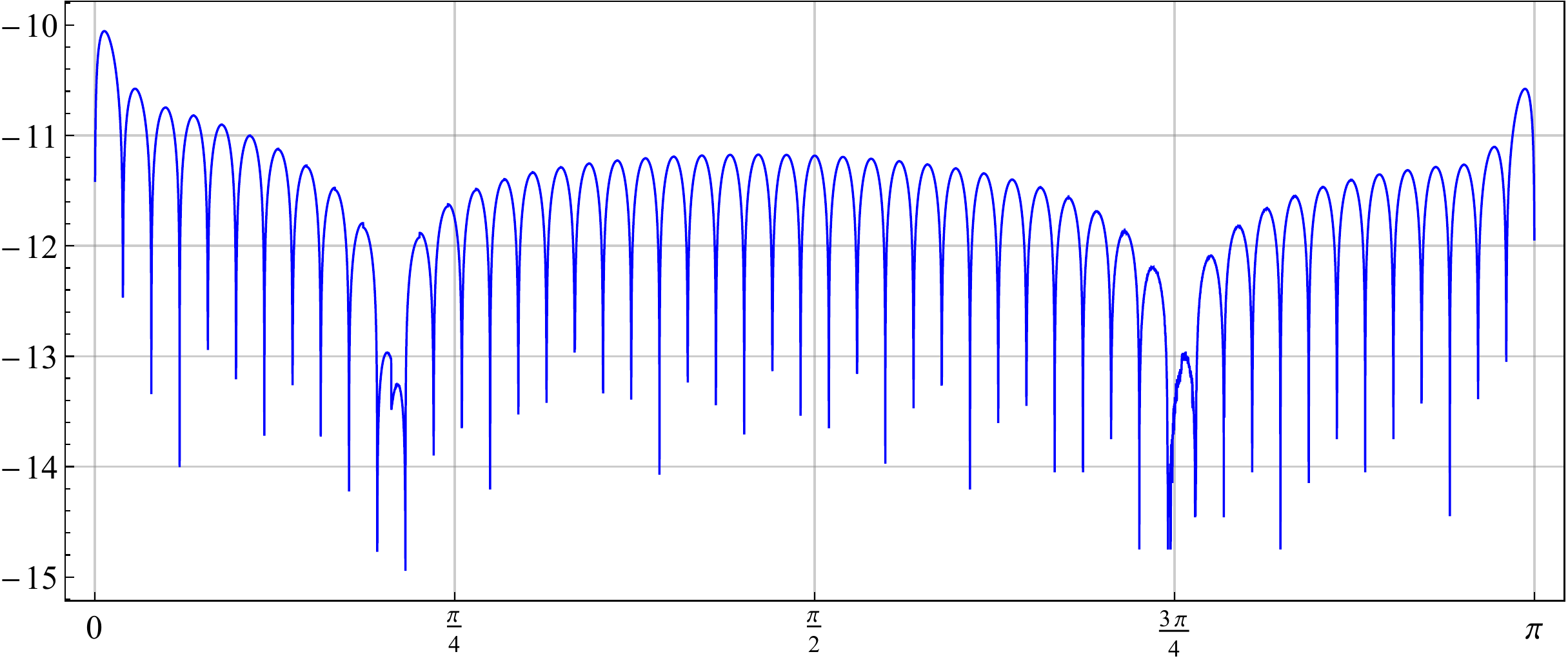}
\bigskip\\
\includegraphics[width=135mm]{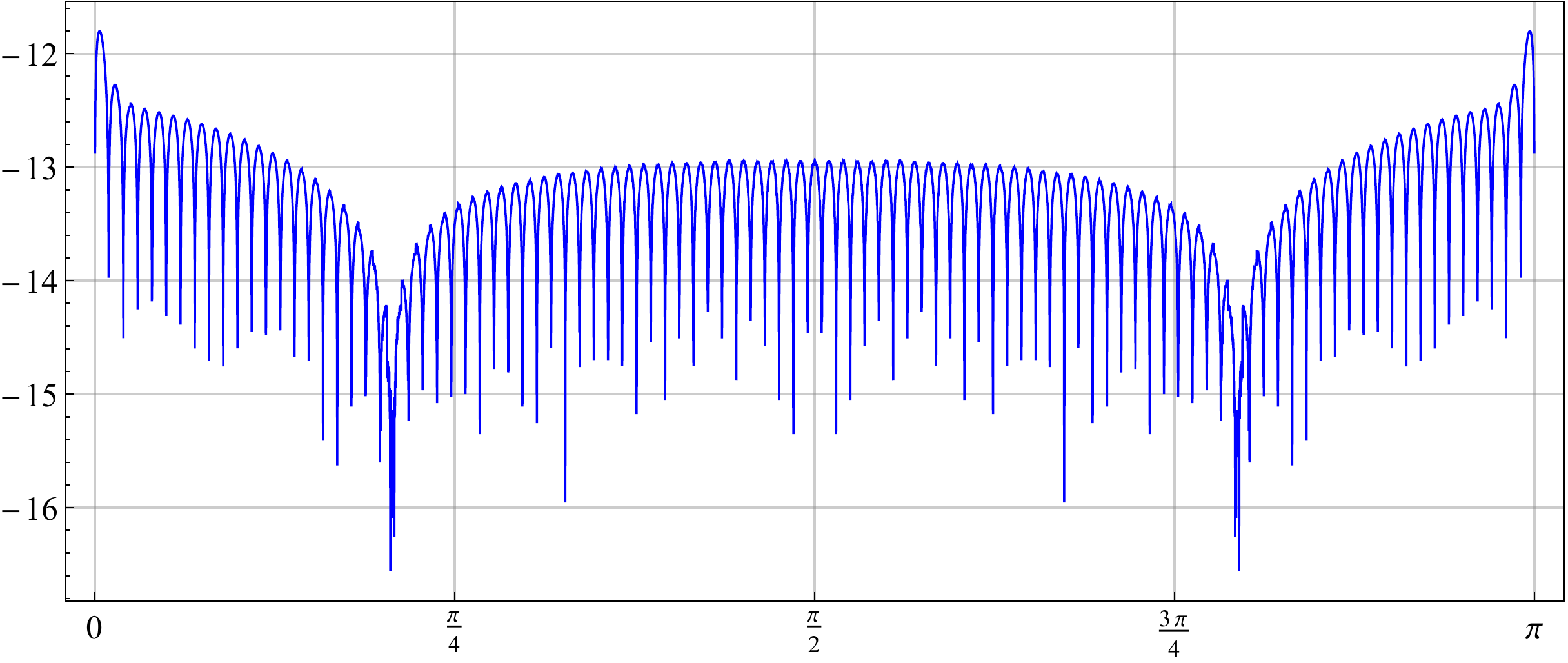}
\caption{The log-scaled eigenvalue errors $\log_{10}(\eps_{j,n}^{(k)})$ for a $k=4$ term expansion and $n=8192$. We worked with a regular mesh of size $n_{0}=25$ (top), $n_{0}=50$ (middle), and $n_{0}=100$ (bottom), and the symbol $(2-2\cos(\tht))^{2}+3(2-2\cos(\tht))h^2+2h^{4}$ (see \eqref{eq:Fn}).}\label{fg:NF+}
\end{figure}

\begin{figure}[ht]
\centering
\includegraphics[width=135mm]{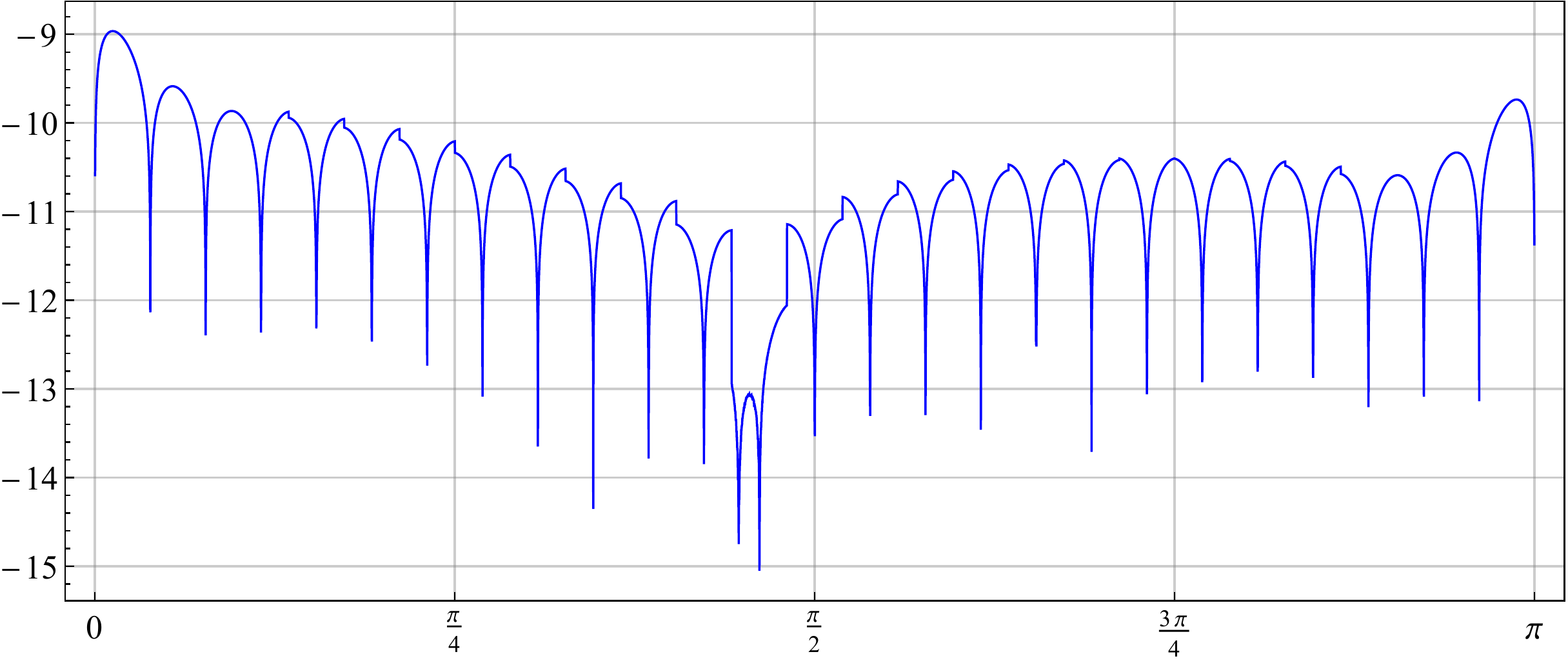}
\bigskip\\
\includegraphics[width=135mm]{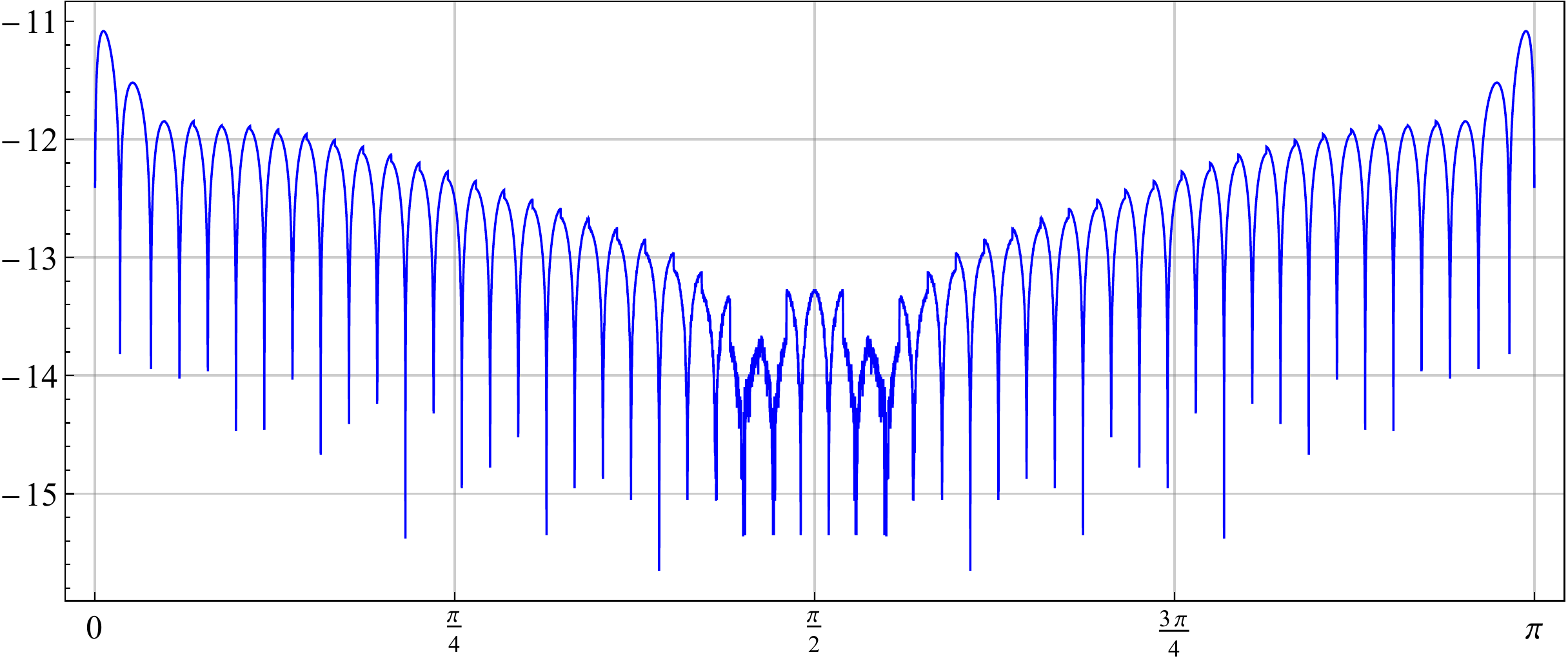}
\bigskip\\
\includegraphics[width=135mm]{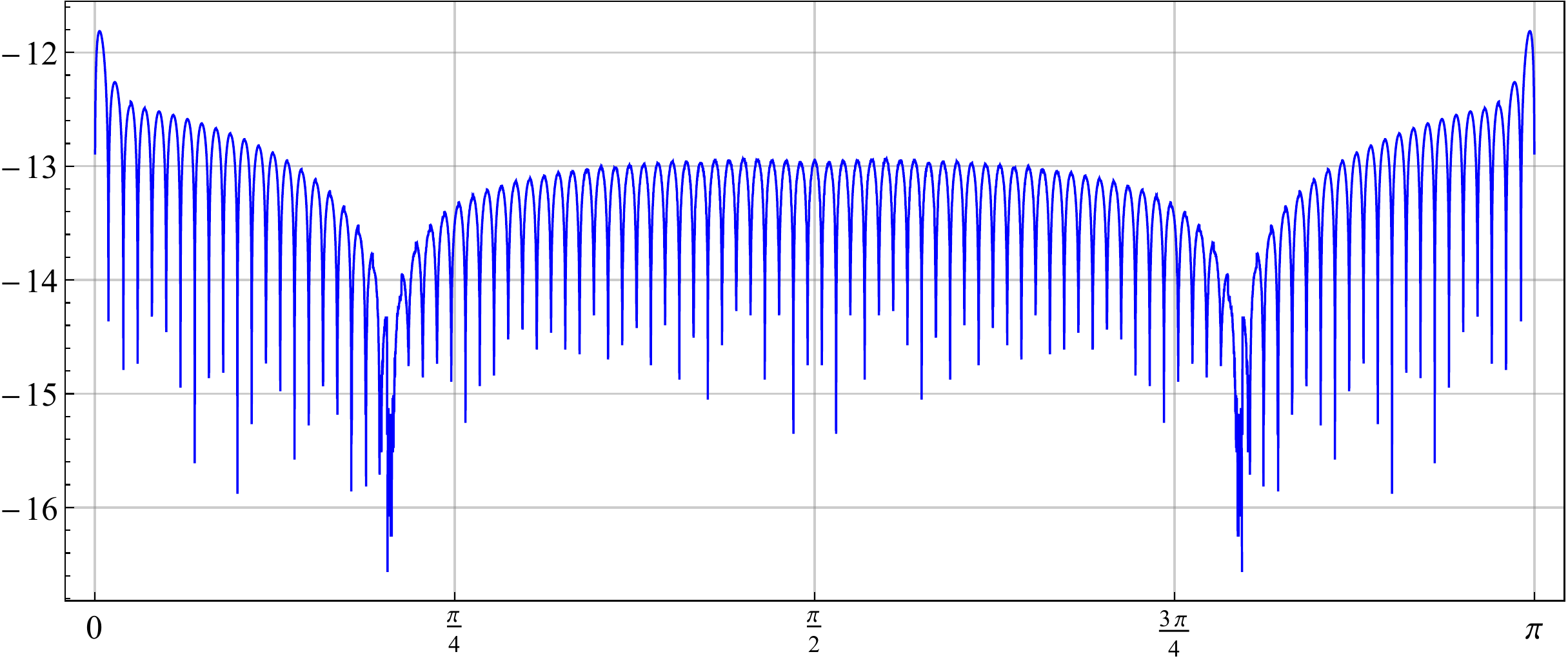}
\caption{
The log-scaled eigenvalue errors $\log_{10}(\eps_{j,n}^{(k)})$ for a $k=4$ term expansion and $n=8192$. We worked with a regular mesh of size $n_{0}=25$ (top), $n_{0}=50$ (middle), and $n_{0}=100$ (bottom), and the symbol $(2-2\cos(\tht))^{2}-3(2-2\cos(\tht))h^2+5h^{4}$ (see \eqref{eq:Fn}).}\label{fg:NF-}
\end{figure}
\clearpage

\begin{table}[ht]
\centering
{\footnotesize\begin{tabular}{|l|l|l|l|l|l|l|}
\hline
\multicolumn{1}{|c|}{$n$} & \multicolumn{1}{|c|}{$256$} & \multicolumn{1}{|c|}{$512$} & \multicolumn{1}{|c|}{$1024$} & \multicolumn{1}{|c|}{$2048$} & \multicolumn{1}{|c|}{$4096$} & \multicolumn{1}{|c|}{$8192$} \\ \hline\hline
$\eps_{n}^{(1)}$ & $1.6179\x10^{\text{-}2}$ & $8.1351\x10^{\text{-}3}$ & $4.0791\x10^{\text{-}3}$ & $2.0424\x10^{\text{-}3}$ & $1.0219\x10^{\text{-}3}$ & $5.1115\x10^{\text{-}4}$\\ \hline
$\hat\eps_{n}^{(1)}$ & $4.1995\x10^{2}$ & $4.2151\x10^{2}$ & $4.2229\x10^{2}$ & $4.2268\x10^{2}$ & $4.2287\x10^{2}$ & $4.2297\x10^{2}$\\ \hline
$\eps_{n}^{(2)}$ & $1.0025\x10^{\text{-}4}$ & $2.5260\x10^{\text{-}5}$ & $6.3398\x10^{\text{-}6}$ & $1.5880\x10^{\text{-}6}$ & $3.9740\x10^{\text{-}7}$ & $9.9398\x10^{\text{-}8}$\\ \hline
$\hat\eps_{n}^{(2)}$ & $2.6283\x10^{2}$ & $1.3219\x10^{2}$ & $6.6289\x10^{1}$ & $3.3193\x10^{1}$ & $1.6609\x10^{1}$ & $8.3074\x10^{0}$\\ \hline
$\eps_{n}^{(3)}$ & $6.7813\x10^{\text{-}8}$ & $8.6059\x10^{\text{-}9}$ & $1.0836\x10^{\text{-}9}$ & $1.3593\x10^{\text{-}10}$ & $1.7021\x10^{\text{-}11}$ & $2.1300\x10^{\text{-}12}$\\ \hline
$\hat\eps_{n}^{(3)}$ & $1.7956\x10^{1}$ & $4.5486\x10^{0}$ & $1.1444\x10^{0}$ & $2.8696\x10^{\text{-}1}$ & $7.1850\x10^{\text{-}2}$ & $1.7980\x10^{\text{-}2}$\\ \hline
$\eps_{n}^{(4)}$ & $5.8948\x10^{\text{-}9}$ & $3.5799\x10^{\text{-}10}$ & $2.5494\x10^{\text{-}11}$ & $5.4580\x10^{\text{-}12}$ & $3.1878\x10^{\text{-}12}$ & $1.5390\x10^{\text{-}12}$\\ \hline
$\hat\eps_{n}^{(4)}$ & $1.5765\x10^{2}$ & $1.9111\x10^{1}$ & $2.7193\x10^{0}$ & $1.1638\x10^{0}$ & $1.3591\x10^{0}$ & $1.3121\x10^{0}$\\ \hline
\end{tabular}}
\vspace{2mm}
\caption{The maximum eigenvalue errors $\eps_{n}^{(k)}$ and normalized errors $\hat\eps_{n}^{(k)}$ for an expansion of $k=1,\ldots,4$ terms, corresponding to the symbol $(2-2\cos(\tht))^{2}-3(2-2\cos(\tht))h^2+5h^{4}$ (see \eqref{eq:Fn}) with different values of $n$ and a regular mesh of $n_{0}=100$ points.}\label{tb:NF-}
\end{table}

\section{Conclusions, Perspectives, and Open Problems}\label{sec:Final}

The eigenvalues of Toeplitz matrices $T_{n}(f)$ with a real-valued symbol $f$, satisfying some conditions and tracing out a simple loop over the interval $[-\pi,\pi]$, are known to admit an asymptotic expansion with the form
\[\la_{j}(T_{n}(f))=f(d_{j,n})+c_{1}(d_{j,n})h+c_{2}(d_{j,n})h^{2}+O(h^{3}),\]
where $h=\frac{1}{n+1}$, $d_{j,n}=\pi j h$, and $c_k$ are some bounded coefficients depending only on $f$.
In practice, the latter expansion is numerically observed under the only condition of monotonicity and even character of the generating function \cite{EkGa18}.

In this note we have investigated the superposition caused over this expansion, when considering a linear combination of symbols and, from a theoretical point of view, we stress that this is the first time that an eigenvalue expansion is theoretically obtained for a Toeplitz matrix-sequence with a symbol depending on $n$. As a further relevant contribution we have improved the precision of the algorithm in \cite{EkGa19}, by using analytic information obtained in our theoretical findings. 

The problem has noteworthy applications in the differential setting, when the coefficients of the linear combination are given functions of $h$:
In particular, by using the new expansions, we can give matrix-less eigensolvers for large matrices stemming from the numerical approximation of standard differential operators and distributed order fractional differential equations.
We notice that the present approach can be viewed also as a successful application of the notion of $\glt$ momentary symbols, a quite new research line discussed in \cite{BoEk21-0,BoEk21a}. As future developments we will consider
\newpage
\begin{itemize}
\item a professional code in the spirit of matrix-less algorithms;
\item an extension of the theory to the case of Hermitian even matrix-valued symbols, using the basic study in \cite{EkFu18b}, with the idea of treating in full generality matrix-sequences stemming from Finite Elements and IgA approximations of coercive differential problems (see also \cite{DoGa15,EkFu18a,GaSp19} and references therein);
\item the case where the monotonicity is violated, which seems to be a very challenging setting, as widely discussed in \cite{EkGa18};
\item the case of variable coefficient differential operators, as a test of the notion of $\glt$ momentary symbols (Definition \ref{def:momentarysymbols}) in full generality: an example of interest in applications would the extension of the techniques to the case of operators as those in equation (\ref{mixed orders}) with $\al_s(x)$ being Riemann integrable functions, for $s=0,1,\ldots,s$. Under these conditions the related matrix-sequences are not of Toeplitz type any longer, but they belong to the $\glt$ class and admit $\glt$ momentary symbols, in accordance with Definition \ref{def:momentarysymbols}.
\end{itemize}

\bibliographystyle{acm}
\bibliography{Toeplitz}
\end{document}